\documentclass{amsart}
\usepackage{amsmath}
\newfont{\cyr}{wncyr10}

\newcommand{\map}[1]{\;\xrightarrow{#1}\;}
\newcommand{\iso}{\cong}
\newcommand{\mil}{\lim\limits_\leftarrow}
\newcommand{\dlim}{\lim\limits_\rightarrow}

\newcommand{\Q}{\mathbf Q}
\newcommand{\Z}{\mathbf Z}
\newcommand{\F}{\mathbf F}
\newcommand{\C}{\mathbf C}
\newcommand{\bk}{\mathbf k}
\newcommand{\cO}{\mathcal O}
\newcommand{\D}{\mathcal{D}}
\newcommand{\bT}{\mathbf T}
\newcommand{\bA}{\mathbf A}
\newcommand{\R}{\mathcal R}

\newcommand{\bmu}{\mathbf \mu}
\newcommand{\mup}{\bmu_{p^\infty}}

\newcommand{\Tw}{\mathrm{Tw}}
\newcommand{\Gal}{\mathrm{Gal}}
\newcommand{\loc}{\mathrm{loc}}
\newcommand{\Hom}{\mathrm{Hom}}
\newcommand{\Aut}{\mathrm{Aut}}

\newcommand{\Ind}{\mathrm{Ind}}
\newcommand{\res}{\mathrm{res}}
\newcommand{\cor}{\mathrm{cor}}
\newcommand{\inv}{\mathrm{inv}}
\newcommand{\unr}{\mathrm{unr}}

\newcommand{\len}{\mathrm{length}}
\newcommand{\tors}{\mathrm{tors}}
\newcommand{\ch}{\mathrm{char}}
\newcommand{\Frob}{\mathrm{Fr}}
\newcommand{\Quot}{\mathrm{Quot}}
\newcommand{\Mod}{\mathrm{Mod}}
\newcommand{\Sel}{\mathrm{Sel}}
\newcommand{\Heeg}{\mathrm{Hg}}
\newcommand{\ord}{\mathrm{ord}}

\newcommand{\relaxed}{\mathrm{rel}}
\newcommand{\strict}{\mathrm{str}}
\newcommand{\f}{\mathrm{f}}
\newcommand{\s}{\mathrm{s}}
\newcommand{\tr}{\mathrm{tr}}
\newcommand{\gr}{\mathrm{gr}}
\newcommand{\Lsel}{{\sel_\Lambda}}
\newcommand{\Fil}{\mathrm{Fil}}

\newcommand{\kolsys}{\mathbf{\kappa}}
\newcommand{\KS}{\mathbf{KS}}
\newcommand{\Tbar}{\bar{T}}
\newcommand{\pl}{\mathcal{L}}
\newcommand{\pn}{\mathcal{N}}
\newcommand{\cH}{\mathcal{H}}
\newcommand{\stub}{\mathcal{S}}

\newcommand{\inert}{\mathcal{I}}

\newcommand{\sel}{\mathcal F}
\newcommand{\msel}{{/\sel}}
\newcommand{\altsel}{\mathcal G}
\newcommand{\gp}{\mathfrak P}
\newcommand{\gq}{\mathfrak Q}
\newcommand{\gm}{\mathfrak m}
\newcommand{\aug}{\mathrm{aug}}

\newcommand{\ga}{\mathfrak{a}}

\newcommand{\Norm}{\mathrm{Norm}}
\newcommand{\bH}{\mathbf{H}}
\newcommand{\pg}{\mathcal{G}}

\input xy
\xyoption{all}

\begin{document}
\title{The Heegner point Kolyvagin system}
\author{Benjamin Howard}
\address{Department of Mathematics\\ Stanford University\\ Stanford, CA\\
94305}
\curraddr{Department of Mathematics\\ Harvard University\\ Cambridge, MA\\
02138}
\thanks{This research is derived from the author's Ph.D. thesis, and was
conducted under the supervision of Karl Rubin. The author 
extends his thanks both to Professor Rubin and to 
the mathematics department of Stanford University.
This research was partially conducted by the author for the Clay
Mathematics Institute, to whom he also gives his thanks.}

\begin{abstract}
In \cite{pr87} Perrin-Riou formulates a form of the Iwasawa main conjecture
which relates Heegner points to the Selmer group of an elliptic curve defined
over $\mathbf{Q}$, as one goes up the anticyclotomic 
$\mathbf{Z}_p$-extension of
a quadratic imaginary field $K$.  Building on the earlier work of Bertolini
on this conjecture, and making use of 
the recent work of Mazur and Rubin on Kolyvagin's theory of Euler systems,
we prove one divisibility of Perrin-Riou's conjectured equality.
As a consequence, one obtains an upper bound on the 
rank of the Mordell-Weil group $E(K)$ in terms of Heegner points.
\end{abstract}

\maketitle
\setcounter{tocdepth}{2}

\theoremstyle{plain}
\newtheorem{Thm}{Theorem}[subsection]
\newtheorem{Prop}[Thm]{Proposition}
\newtheorem{Lem}[Thm]{Lemma}
\newtheorem{Cor}[Thm]{Corollary}
\newtheorem{Theorem}{Theorem}

\theoremstyle{definition}
\newtheorem{Def}[Thm]{Definition}

\theoremstyle{remark}
\newtheorem{Rem}[Thm]{Remark}

\renewcommand{\labelenumi}{(\alph{enumi})}
\renewcommand{\theTheorem}{\Alph{Theorem}}
\setcounter{section}{-1}

\section{Introduction}

In this paper we  modify the notion of a Kolyvagin system, as
defined in \cite{mazur-rubin}, to include the system of cohomology classes
which result from the application of Kolyvagin's derivative operators to
the Heegner point Euler system.  The resulting theory yields a
simplified proof of a theorem of Kolyvagin, stated below as Theorem
\ref{kolyvagin's theorem}.
Our true sights, however, are set on the Iwasawa theory of Heegner points
in the anticyclotomic $\Z_p$-extension of a quadratic imaginary field.

Fix forever a rational prime $p$.
If $E$ is an elliptic curve defined over a number field $L$, we denote
by $\Sel_{p^\infty}(E/L)$ and $S_p(E/L)$ the usual $p$-power Selmer
groups which fit into the descent sequences
$$0\map{}E(L)\otimes\Q_p/\Z_p\map{}\Sel_{p^\infty}(E/L)\map{}
\mbox{\cyr Sh}_{p^\infty}\map{}0$$
$$0\map{}E(L)\otimes\Z_p\map{}S_p(E/L)\map{}\mil
\mbox{\cyr Sh}_{p^n}\map{}0.$$

Fix once and for all an elliptic curve  $E/\Q$ with conductor $N$
and a quadratic imaginary field $K$ of discriminant 
$D\not=-3,-4$ satisfying the Heegner hypothesis that all primes 
dividing $N$ are split in $K$. Let
$T=T_p(E)$ be the $p$-adic Tate module of $E$.
The theory of complex
multiplication gives a family of points on the modular curve $X_0(N)$
which are rational over abelian extensions of $K$.  More precisely,
in Section \ref{Heegner points} we will attach 
to every squarefree product $n$ of rational primes inert in $K$
a point $h_n\in X_0(N)(K[n])$, where $K[n]$ is the ring
class field of $K$ of conductor $n$.
Fixing a modular parametrization of $E$ by $X_0(N)$
yields a family of points $P[n]\in E(K[n])$
which satisfy Euler system-like relations relative to the norm operators.
To each point $P[n]$ one applies first the Kummer map and then
Kolyvagin's derivative operator $D_n$ to obtain a cohomology class over $K$,
$$\kappa_n\in H^1_{\sel(n)}(K,T/I_nT)\otimes G_n$$
where $G_n=\bigotimes_{\ell|n}\Gal(K[\ell]/K[1])$,
$I_n$ is an ideal of $\Z_p$, and $$H^1_{\sel(n)}(K,T/I_nT)\subset
H^1(K,T/I_nT)$$ is the generalized Selmer group of Definition 
\ref{modified selmer}
obtained by modifying the usual local conditions which define
$S_p(E/K)$ at primes of $K$ dividing $n$. The classes $\kappa_n$ form
a Kolyvagin system, as defined in Section \ref{ks section}.
The class $\kappa_1\in
H^1_{\sel(1)}(K,T)=S_p(E/K)$ is just the image under the Kummer map
of the norm of $P[1]$, and
the celebrated theorem of Gross and Zagier says that
$\ord_{s=1}L(s,E/K)=1$ iff  $\kappa_1$ has infinite order.
In Section \ref{KS} we will give a proof of the following theorem.

\begin{Theorem}\label{kolyvagin's theorem}(Kolyvagin)
Assume $p$ is odd and the integers $p$, $D$, and $N$
are pairwise coprime.  Assume also that 
$\Gal(\bar{K}/K)\map{}\Aut_{\Z_p}(T)$ is surjective.
If $\kappa_1\not=0$ then $S_p(E/K)$ is free of rank one
over $\Z_p$ and there is a finite $\Z_p$-module
$M$ such that $$\Sel_{p^\infty}(E/K)\iso
(\Q_p/\Z_p)\oplus M\oplus M$$ with $$\len_{\Z_p}(M)\le \len_{\Z_p}(
S_{p}(E/K)/\Z_p\kappa_1).$$
\end{Theorem}

Assume now that $E$ is ordinary at $p$.
Let $K_\infty$ be the anticyclotomic $\Z_p$-extension of $K$,
$\Gamma=\Gal(K_\infty/K)$, and $\Lambda=\Z_p[[\Gamma]]$.
Let $K_n\subset K_\infty$ be the unique subfield with $[K_n:K]=p^n$.
In Section \ref{ks lambda section} we define, in the manner of
\cite{greenberg-representations}, two generalized Selmer groups
$$H^1_\Lsel(K,\bT)\subset \mil H^1(K_n,T)\hspace{1cm}
H^1_\Lsel(K,\bA)\subset \dlim H^1(K_n,E[p^\infty]),$$ where
$\bT\iso T\otimes\Lambda$ and $\bA\iso\Hom(\bT,\mup)$, such that
there are pseudo-isomorphisms of $\Lambda$-modules
$$ H^1_\Lsel(K,\bT)\sim\mil S_{p}(E/K_n)\hspace{1cm}
H^1_\Lsel(K,\bA)\sim\dlim\Sel_{p^\infty}(E/K_n).$$ Define
$X=\Hom(H^1_\Lsel(K,\bA),\Q_p/\Z_p)$, and let $X_{\Lambda-\tors}$
denote the $\Lambda$-torsion submodule of $X$. In the spirit of
the Iwasawa Main Conjecture we view the characteristic ideal
$\ch(X_{\Lambda-\tors})$ as a sort of algebraically defined
$p$-adic $L$-function.

In Section \ref{iwasawa heegner points}
we use Heegner points to construct a  Kolyvagin
system $\kappa^\Heeg$ for the $\Lambda$-module $\bT$.  The class
$\kappa_1^\Heeg\in H^1_\Lsel(K,\bT)$ is nonzero by the work of Cornut
and Vatsal.
At a height-one prime $\gp$ of $\Lambda$, a Kolyvagin system
for $\bT$ reduces to  a Kolyvagin system
for $\bT\otimes_\Lambda S_\gp$ where $S_\gp$ is the integral closure
of $\Lambda/\gp$. Applying at every prime of $\Lambda$
the same machinary used to prove Theorem \ref{kolyvagin's theorem}
gives the following result.

\begin{Theorem}\label{Big Money} 
Keep the assumptions on $T$, $p$, $D$, and $N$ of Theorem 
\ref{kolyvagin's theorem},
and assume also that $p$ does not divide the class number of $K$.
We continue to assume that $E$ is ordinary at $p$.
Let $\bH$ denote the $\Lambda$-submodule
of $H^1_\Lsel(K,\bT)$ generated by $\kappa_1^\Heeg$, and let
$\iota:\Lambda\map{}\Lambda$ be the involution induced by inversion in
$\Gamma$.  

The $\Lambda$-module $H^1_\Lsel(K,\bT)$ is torsion-free of
rank one, and there is a finitely-generated torsion $\Lambda$-module
$M$ such that  \begin{enumerate}
\item $\ch(M)=\ch(M)^\iota$
\item $X\sim\Lambda\oplus M\oplus M$
\item
$\ch(M)$ divides $\ch\big(H^1_\Lsel(K,\bT)/\bH\big)$
\end{enumerate}
where $\ch$ denotes characteristic ideal.
\end{Theorem}

We remark that parts (a) and (b) are already known by the combined results
of Bertolini, Cornut, and Nekov\'{a}\v{r} \cite{bertolini, cornut, nek} 
and have the following important consequence: by Mazur's control 
theorem one has $$\mathrm{rank}_{\Z_p}X/(\gamma-1)X=
\mathrm{corank}_{\Z_p}\Sel_{p^\infty}(E/K),$$ and therefore the
corank of the Selmer group over $K$ odd. This is compatible with
the Birch and Swinnerton-Dyer conjecture: the Heegner hypothesis
forces the sign of the functional equation of $L(s,E/K)$ to be
$-1$, and so $\ord_{s=1}L(s,E/K)$ is odd.  Similarly, part (c) 
of the theorem,  together with the control theorem, gives the inequality
\begin{equation}\label{intro bound}
\mathrm{rank}_{\Z_p}S_p(E/K) \le 1+2\cdot \ord_J(\mathbf{L})
\end{equation}
where $J\subset\Lambda$ is the augmentation ideal and 
$\mathbf{L}=\ch\big(H^1_\Lsel(K,\bT)/\bH\big)$.  One does not typically
expect equality to hold;  see (\ref{height degeneracy}) below.
Theorem \ref{Big Money} can be generalized in
many ways, for example by replacing $E$ by an abelian variety
with real multiplication, replacing the modular curve $X_0(N)$
by an appropriate Shimura curve (allowing one to weaken the Heegner
hypothesis), and replacing $K$ by a CM-field.  See \cite{howard-shimura}
for work in this direction.

The Main Conjecture for Heegner points was formulated by Perrin-Riou
in \cite{pr87} and predicts that
$$\ch(M)=c^{-1}\cdot\ch\big(H^1_\Lsel(K,\bT)/\bH\big)$$
where $c\in\Z_p$ is the Manin constant associated to our choice of
modular parametrization of $E$ (the proof that our $\bH$ agrees with
the module considered by Perrin-Riou is part of the content
of Theorem \ref{heegner module}).
The theory of derived $p$-adic height pairings, introduced by 
Bertolini and Darmon and further developed by the author
\cite{bertdarL, howard-heights}, leads one
to conjecture that the the torsion module $M$ of
Theorem \ref{Big Money} has the form 
$$M\sim (\Lambda/J)^{e_1}\oplus (\Lambda/J^2)^{e_2}\oplus M'$$
for a $\Lambda$-module $M'$ with characteristic ideal prime to $J$, and
$$
e_1=\min(r^+,r^-)\hspace{1cm}
e_2=\frac{|r^+-r^-|-1}{2}$$
where $r^\pm$ is the rank of the $\pm$-eigenspace of $S_p(E/K)$
under complex conjugation.
Combining this with the Main Conjecture, we see that one should 
expect 
\begin{equation}\label{height degeneracy}
 \ord_J(\mathbf{L})=e_1+2e_2=\max(r^+,r^-)-1.
\end{equation}
Since the left hand side of (\ref{intro bound}) is $1+2e_1+2e_2$ by
Mazur's control theorem, one expects equality to hold there
exactly when $e_2=0$.

\bigskip
The following conventions will remain in effect throughout.
By a coefficient ring, $R$, we mean a
complete, Noetherian, local ring with finite residue field of characteristic
$p$.  The cases of interest are when $R$ is the ring of integers $\cO$ of
a finite extension of $\Q_p$, a quotient of $\cO$, or the Iwasawa
algebra $\Lambda$. The maximal ideal of $R$ is denoted $\gm$.
We denote by $R(1)$ the Tate twist of $R$, i.e. the free rank-one $R$-module
on which Galois acts through the cyclotomic character.

If $M$ is any $R$-module and $I\subset R$ is an ideal then
$M[I]$ is the submodule of $M$ consisting of elements annihilated by
every $r\in I$.  We define $M(1)=M\otimes_R R(1)$.
If $L$ is a perfect field (which is all we shall ever have need to consider),
then $\bar{L}$ denotes the algebraic closure of $L$, and $G_L=\Gal(\bar{L}/L)$.
If $L$ is a local field we let $L^\unr$ denote the maximal unramified
extension of $L$ and denote by $\Frob$ the Frobenius automorphism of
$L^\unr/L$.

\section{Kolyvagin systems}\label{KS}

Throughout Section \ref{KS} we fix a coefficient ring $R$ and a quadratic
imaginary field $K$.  If $L$ is a perfect field, we denote by
$\Mod_{R,L}$ the category of
finitely-generated $R$-modules equipped with continuous, linear actions of
$G_L$, assumed to be unramified outside of a finite set of primes
in the case where $L$ is a global field.  The letter $T$
will always denote an object of this category (for some field $L$).
Let $\tau\in G_\Q$ be a fixed complex conjugation.

Sections \ref{selmer groups}, \ref{ks section}, and \ref{hypotheses}
follow \cite{mazur-rubin} very closely.  
Sections \ref{Artinian} and \ref{bound}
do as well, but with modifications unique to the case of Heegner points.
The results of Section \ref{generalized cassels}, which rely crucially
on the self-duality of the Tate module $T_p(E)$, have no analogue in
\cite{mazur-rubin}.


\subsection{Selmer groups}\label{selmer groups}


Fix a finite place $v$ of $K$,
and denote by $\inert_v$ the inertia subgroup of $G_{K_v}$,
$\Frob_v\in \Gal(K_v^\unr/K_v)$ the Frobenius element,
and $\bk_v$ the residue field of $K_v$.  Let $T$ be an object of
$\Mod_{R,K_v}$.

\begin{Def}
A \emph{local condition} on $T$ (over $K_v$) is a choice of $R$-submodule
of $H^1(K_v,T)$.  We will frequently use $\sel$ to denote a local condition,
in which case the submodule will be denoted
$H^1_\sel(K_v,T)\subset H^1(K_v,T).$

Given an  $R[[G_{K_v}]]$-submodule (resp. quotient) $S$ of $T$ and a local
condition $\sel$ on $T$ we define the \emph{propagated} condition, still
denoted by $\sel$, on $S$ to be the preimage (resp. image)  of $H^1_\sel
(K_v,T)$ under the natural map $$H^1(K_v,S)\map{}H^1(K_v,T)$$ (resp.
$H^1(K_v,T)\map{}H^1(K_v,S)$).
\end{Def}

We will be concerned primarily (but not entirely) with local conditions
of the following types. \begin{enumerate}
\item The \emph{relaxed} and \emph{strict} conditions (respectively)
$$H^1_{\relaxed}(K_v,T)=H^1(K_v,T)\hspace{1cm}H^1_\strict(K_v,T)=0,$$
\item the \emph{unramified} condition $$H^1_\unr(K_v,T)=\mathrm{ker}\big(
H^1(K_v,T)\map{}H^1(K_v^\unr,T)\big),$$
\item the \emph{$L$-transverse} condition
$$H^1_{L-\tr}(K_v,T)=\mathrm{ker}
\big(H^1(K_v,T)\map{}H^1(L,T)\big)$$ where $K_v$
has residue characteristic $\not=p$ and $L$ is a
maximal totally tamely ramified abelian
$p$-extension of $K_v$.
\end{enumerate}

If $K_v$ has residue characteristic
different from $p$ and  $T$  is unramified (i.e. the inertia
group $\inert_v$ acts trivially on $T$), then
we shall also refer to the unramified
condition on $T$ as the \emph{finite} condition $H^1_\f(K_v,T)$.
We then define the $\emph{singular quotient}$ $H^1_\s(K_v,T)$ by
exactness of $$0\map{}H^1_\f(K_v,T)\map{}H^1(K_v,T)
\map{}H^1_\s(K_v,T)\map{}0.$$

If $\mathcal{T}$ is
a subcategory of  $\Mod_{R,{K_v}}$ then by a \emph{local condition functorial
over $\mathcal{T}$} we mean a subfunctor of $H^1(K_v,\ \ )$,
$$T\mapsto H^1_\sel(K_v,T)\subset H^1(K_v,T).$$
The local conditions defined above are all functorial over $\Mod_{R,{K_v}}$.

\begin{Def}\label{cartesian}
A local condition $\sel$ functorial over a subcategory $\mathcal{T}$ of
$\Mod_{R,{K_v}}$ is \emph{cartesian} if for any injective morphism $\alpha:
S\map{}T$ the local condition $\sel$ on $S$ is the same as the local condition
obtained by propagating $\sel$ from $T$ to $S$.
\end{Def}

\begin{Def}\label{Quot}
For $T$ an object of $\Mod_{R,{K_v}}$ we define the \emph{quotient category}
of $T$ $\Quot(T)$ to be the category whose objects are quotients $T/IT$
of $T$ by ideals of $R$ and the morphisms from $T/IT$ to $T/JT$ are the
maps induced by scalar multiplications $r\in R$ with $rI\subset J$.
\end{Def}

Any local condition on $T$ defines a local condition functorial over
$\Quot(T)$ by propagation.

\begin{Rem}\label{cartesian identifications}
Of special interest is the case where $R$ is principal and Artinian
of length $k$ and $T$ is a free $R$-module.  Let $\gm=\pi R$ be the
maximal ideal of $R$.
A local condition on $\Quot(T)$ being cartesian implies that for
$i<k$ the local condition on the submodule $T[\gm^i]$ (propagated from $T$)
agrees with the local condition on $T/\gm^i T$ when the two
modules are identified via the isomorphism
$$T/\gm^i T\map{\pi^{k-i}}T[\gm^i].$$
\end{Rem}

\begin{Lem}\label{unramified cartesian}
The unramified local condition is cartesian on any
subcategory of $\Mod_{R,{K_v}}$
whose objects are unramified $G_{K_v}$-modules.
\end{Lem}
\begin{proof}
This is Lemma 1.1.9 of \cite{mazur-rubin}.
\end{proof}

\begin{Def}\label{local duality}
Set $T^*=\Hom(T,R(1))$.
We give $T^*$ the structure of a $G_{K_v}$-module
by letting $\sigma\in G_{K_v}$ act on $f(t)$ by $f(t)\mapsto
\sigma f(\sigma^{-1}t)$.
Local Tate duality gives a perfect $R$-bilinear pairing
$$\langle \ ,\ \rangle :H^1(K_v,T)\times H^1(K_v, T^*)
\map{}H^2(K_v,R(1))\map{\inv}R$$
and for any local condition $\sel$ on $T$ we define the \emph{dual local
condition},  $\sel^*$, on $T^*$ to be orthogonal
complement of $\sel$ under the above  local pairing.
\end{Def}

\begin{Prop} \label{locally free} Assume that $v$ does not divide $p$,
$T$ is unramified at $v$,
and that $|\bk_v^\times|\cdot T=0$. There are canonical
isomorphisms $$H^1_\f(K_v,T)\iso T/(\Frob_v-1)T\hspace{1cm}
H^1_\s(K_v,T)\otimes\bk_v^\times\iso T^{\Frob_v=1}.$$
\end{Prop}
\begin{proof}This is Lemma 1.2.1 of \cite{mazur-rubin}.
The first map is given on cocycles
by evaluation at the Frobenius automorphism, and the second by
$c\otimes \alpha\mapsto c(\sigma_\alpha)$ where $\sigma_\alpha\in
\Gal(K_v^\mathrm{ab}/K_v^\unr)$ is the Artin symbol of any lift of $\alpha$
to $K$.
\end{proof}

\begin{Def}\label{finite singular}
If $v$ does not divide $p$, $G_{K_v}$ acts trivially on $T$,
and $|\bk_v^\times|\cdot T=0$,
we define the \emph{finite-singular comparison map}  to be the  isomorphism
 $$\phi^{\f\s}_v:H^1_\f(K_v,T)\iso T\iso H^1_\s(K_v,T)\otimes \bk_v^\times$$
given by Proposition \ref{locally free}.
\end{Def}

\begin{Prop}\label{splitting}\label{local isotropy}
Keep the assumptions of Definition \ref{finite singular}.  We fix a
maximal totally tamely ramified abelian $p$-extension  $L/K_v$, and
hence a choice of $L$-transverse condition on $T$.
The transverse submodule $H^1_\tr(K_v,T)$ projects isomorphically onto
$H^1_\s(K_v,T)$ giving a splitting $$H^1(K_v,T)= H^1_\f(K_v,T)\oplus
H^1_\tr(K_v,T).$$  Furthermore, under the local Tate pairing
\begin{enumerate}
\item $H^1_\f(K_v,T)$ and $H^1_\f(K_v,T^*)$ are exact orthogonal complements,
\item $H^1_\tr(K_v,T)$ and $H^1_\tr(K_v,T^*)$ are exact orthogonal complements.
\end{enumerate}
\end{Prop}
\begin{proof}These statements are  Lemma 1.2.4 and Proposition 1.3.2
of \cite{mazur-rubin}.
\end{proof}

\bigskip
We now consider global cohomology groups.  Fix an object $T$ of
$\Mod_{R,K}$.

\begin{Def} By a \emph{Selmer structure} $\sel$ on $T$ (over $K$) we mean
a finite set of places $\Sigma(\sel)$ of $K$ containing $p$,
all archimedean places, and
all places at which $T$ is ramified, and for each $v\in\Sigma(\sel)$
a choice of local condition $H^1_\sel(K_v,T)$.
Given a Selmer structure $\sel$ on $T$ we define the associated \emph{Selmer
module} $$H^1_\sel(K,T)\subset H^1(K,T)$$ to be the kernel of
$$H^1(K_{\Sigma(\sel)}/K,T)\map{}\bigoplus_{v\in\Sigma(\sel)}
H^1(K_v,T)/H^1_\sel(K_v,T)$$ where $K_{\Sigma(\sel)}$ is the maximal
extension of $K$ unramified away from the places of $\Sigma(\sel)$.
\end{Def}

Given a Selmer structure $\sel$ we will usually write $H^1_\sel(K_v,T)$
for $H^1_\f(K_v,T)$ for a prime $v\not\in\Sigma(\sel)$. Then $H^1_\sel(K,T)$
is nothing more than the set of classes in $H^1(K,T)$ whose
localization lives in $H^1_\sel(K_v,T)$ at every place $v$.
There is a natural partial ordering on the set of all Selmer structures,
namely we write $\sel\le\altsel$ iff $H^1_{\sel}(K_v,T)\subset H^1_{\altsel}
(K_v,T)$ for every place $v$ of $K$. Clearly if $\sel\le\altsel$ we
have $H^1_{\sel}(K,T)\subset H^1_{\altsel}(K,T)$.
If $\sel$ is a Selmer structure on $T$ then the collection of dual local
conditions gives a Selmer structure $\sel^*$ on $T^*$ with $\Sigma(T)=
\Sigma(T^*)$.  The following theorem is the fundamental tool
which turns Kolyvagin systems into bounds on Selmer groups.

\begin{Thm}(Poitou-Tate global duality)\label{global duality}
Suppose $\sel\le\altsel$ are Selmer structures on $T$.  There are
exact sequences
\begin{eqnarray*}& &0\map{}H^1_\sel(K,T)\map{}H^1_\altsel(K,T)\map{\loc}
\bigoplus_v H^1_\altsel(K_v,T)/H^1_\sel(K_v,T)\\ & &
0\map{}H^1_{\altsel^*}(K,T^*)\map{}H^1_{\sel^*}(K,T^*)\map{\loc}
\bigoplus_v H^1_{\sel^*}(K_v,T^*)/H^1_{\altsel^*}(K_v,T^*)\end{eqnarray*}
and the images of the rightmost arrows are exact orthogonal complements under
the sum of the local pairings of Definition \ref{local duality}.
\end{Thm}
\begin{proof}See \cite{milne} I.4.10 or \cite{rubin} 1.7.3.
\end{proof}


\subsection{Kolyvagin systems}\label{ks section}


Let $T$ be an object of $\Mod_{R,K}$, and denote by
$\pl_0=\pl_0(T)$ the set of degree two primes of $K$ which do not
divide $p$ or any prime at which $T$ is ramified.  We will
consistently confuse a prime of $\pl_0$ with the rational prime
below it, and if the distinction needs to be made we will write
$\ell\mid\lambda\in\pl_0$ to indicate that $\ell$ is the rational
prime and $\lambda$ the prime of $K$.
\begin{Def}\label{kolyvagin primes}\

\begin{enumerate}
\item For each $\ell\mid\lambda\in\pl_0$, define $I_\ell$ to be the
smallest ideal of $R$ containing $\ell+1$ for which
$\Frob_\lambda$ acts trivially on $T/I_\ell T$.
\item
For every $k\in\Z^+$ define $\pl_k=\pl_k(T)=\{\ell\in\pl_0\mid
I_\ell\subset p^k\Z_p\}.$
\item For $\ell\mid \lambda\in\pl_0$ let $G_\ell=\bk_\lambda^\times/
\bk_\ell^\times$
where $\bk_\ell$ and $\bk_\lambda$ are the residue fields of $\ell$ and
$\lambda$, respectively.
\item Let $\pn_k$ denote the set of squarefree products of primes of $\pl_k$.
For $n\in\pn_0$ define $$I_n=\sum_{\ell\mid n}I_\ell\subset R\hspace{1cm}
G_n=\bigotimes_{\ell\mid n}G_\ell.$$ By convention
$1\in\pn_k$ for every $k$, $I_1=0$, and $G_1=\Z$.
\end{enumerate}
\end{Def}

For $\ell\mid\lambda\in\pl_0$ we denote by $K[\ell]$ the ring
class field of conductor $\ell$.  Since $\lambda$ splits
completely in the Hilbert class field of $K$, the maximal
$p$-subextension of the local extension
$K[\ell]_\lambda/K_\lambda$ (call it $L$) is a maximal totally
tamely ramified abelian $p$-extension of $K_\lambda$ whose Galois
group is canonically identified with the $p$-Sylow subgroup of
$G_\ell$ by class field theory. We therefore have for such a
$\lambda$ a canonical choice of $L$-transverse condition as in
Section \ref{selmer groups}, which we denote by
$H^1_\tr(K_\ell,T)$.

By a \emph{Selmer triple} $(T,\sel,\pl)$ we mean an object $T$ of
$\Mod_{R,K}$, a choice of Selmer structure $\sel$ on $T$, and a
(typically infinite) subset $\pl\subset \pl_0$ which is disjoint
from $\Sigma(\sel)$. We define $\pn=\pn(\pl)$ to be the set of
squarefree products of primes of $\pl$, with the convention that
$1\in \pn(\pl)$.

\begin{Def}\label{modified selmer}
Given a Selmer triple $(T,\sel,\pl)$ and $abc\in\pn(\pl)$ we define a
new Selmer triple $(T,\sel^a_b(c), \pl(abc))$ by taking $\Sigma(\sel^a_b(c))$
to be $\Sigma(\sel)$ together with all prime divisors of $abc$, and taking
$\pl(abc)$ to be $\pl$ with all prime divisors of $abc$ removed.
At any place $\lambda$ of $K$ define the local condition
$\sel^a_b(c)$ to be
$$H^1_{\sel^a_b(c)}(K_\lambda,T)=\left\{\begin{array}{ll}
H^1_\relaxed(K_\lambda,T)& \mathrm{if\ } \lambda\mid a  \\
H^1_\strict(K_\lambda,T)& \mathrm{if\ } \lambda\mid b  \\
H^1_\tr(K_\lambda,T)& \mathrm{if\ } \lambda\mid c
\end{array}\right.$$
and retain the original local condition
$$H^1_{\sel^a_b(c)}(K_\lambda,T)=H^1_\sel(K_\lambda,T)$$
if $\lambda$ does not divide $abc$. If any one of $a$,
$b$, or $c$ is $1$ we omit it from the notation.
\end{Def}

For any $n\ell\in\pn_0$, we may identify the $p$-Sylow subgroups
of $G_\ell$ and $\mathbf{k}_\lambda^\times/\mathbf{k}_\ell^\times$ 
via the Artin symbol,
and let $$\phi_\ell^{\f\s}:H^1_\f(K_\ell,T/I_{n\ell}T) \iso
H^1_\s(K_\ell,T/I_{n\ell}T)\otimes G_{\ell} $$ be the
finite-singular comparison map at $\ell$.  We have maps

\begin{equation}\label{ks relations}\xymatrix{
 &H^1_{\sel(n)}(K,T/I_n T)\otimes G_n \ar[d]^{\loc_\ell}\\
 &H^1_\f(K_\ell,T/I_{n\ell}T)\otimes G_n
\ar[d]^{\phi_\ell^{\f\s}\otimes 1} \\
H^1_{\sel(n\ell)}(K,T/I_{n\ell}T)\otimes G_{n\ell} \ar[r]^{\loc_\ell}&
H^1_\s(K_\ell,T/I_{n\ell}T)\otimes G_{n\ell}.}\end{equation}

\begin{Def} Given a Selmer triple $(T,\sel,\pl)$
we define a \emph{Kolyvagin system} $\kolsys$ for $(T,\sel,\pl)$
to be a collection of cohomology classes $$\kappa_n\in
H^1_{\sel(n)}(K,T/I_nT)\otimes G_n,$$ one for each $n\in\pn(\pl)$,
such that for any $n\ell\in\pn(\pl)$ the images of $\kappa_n$ and
$\kappa_{n\ell}$ in $H^1_\s(K_\ell,T/I_{n\ell}T)\otimes G_{n\ell}$
under the maps of (\ref{ks relations}) agree.  We denote the
$R$-module of all Kolyvagin systems for $(T,\sel,\pl)$ by
$\KS(T,\sel,\pl)$.
\end{Def}

\begin{Rem}\label{functorality}
The module of Kolyvagin systems has the following  functorial properties:
\begin{enumerate}
\item if $\pl'\subset \pl$ then there is a map
$\KS(T,\sel,\pl)\map{}\KS(T,\sel,\pl'),$
\item if $H^1_\sel(K_v,T)\subset H^1_\altsel(K_v,T)$ at every place $v$
then there is a map
$$\KS(T,\sel,\pl)\map{}\KS(T,\altsel,\pl),$$
\item if $R\map{}R'$ is a ring homomorphism then there is a map
$$\KS(T,\sel,\pl)\otimes_R R'\map{}\KS(T\otimes_R R',\sel\otimes_R R',\pl)$$
where the local condition $\sel\otimes_R R'$ is defined as the image of
$$H^1_\sel(K_v,T)\otimes_R R'\map{}H^1(K_v,T\otimes_R R')$$ for
$v\in\Sigma(\sel)$, and $\Sigma(\sel\otimes_R R')=\Sigma(\sel)$.
\end{enumerate}

\end{Rem}


\subsection{Hypotheses}
\label{hypotheses}


In this subsection $R$ is a coefficient ring and $T$ is an object of
$\Mod_{R,G_K}$. The maximal ideal of $R$ is denoted $\gm$, and
$\Tbar= T/\gm T$ is the residual representation of $T$.  We denote
by $\Tw(T)$ denote the $G_K$-module whose underlying $R$-module is
$T$ and on which $G_K$ acts through the automorphism conjugation
by $\tau$. The identity map on the underlying $R$-modules
$T\map{}\Tw(T)$ and the automorphism of $G_K$ given by conjugation
by $\tau$ induce a ``change of group'' $(G_K,T)\leadsto
(G_K,\Tw(T))$ which induces an isomorphism on cohomology
$$H^i(K,T)\iso H^i(K,\Tw(T)).$$  Similarly at any place $v$ of $K$
conjugation by $\tau$ induces an isomorphism
$$H^i(K_{\bar{v}},T)\iso H^i(K_v,\Tw(T))$$ where $\bar{v}=v^\tau$.

We fix a Selmer triple
$(T,\sel,\pl)$ and record some desirable hypotheses which it
may satisfy:

\renewcommand{\labelenumi}{H.\arabic{enumi}}
\begin{enumerate}
\setcounter{enumi}{-1}

\item 
$T$ is a free, rank 2 $R$-module.

\item
$\Tbar$ is an absolutely irreducible representation of $(R/\gm)[[G_K]]$.

\item
There is a Galois extension $F/\Q$ such that $K\subset F$, $G_F$ acts 
trivially on $T$, and $$H^1(F(\mup)/K, \Tbar)=0.$$

\item
For every $v\in\Sigma(\sel)$ the local condition $\sel$ at $v$
is cartesian on the category $\Quot(T)$ (see Definitions \ref{cartesian}
and \ref{Quot}).

\item
There is a perfect, symmetric,  $R$-bilinear pairing
$$(\ ,\ ): T\times T\map{}R(1)$$ which satisfies
$(s^\sigma,t^{\tau\sigma\tau^{-1}})=(s,t)^\sigma$ for
every $s,t\in T$ and $\sigma\in G_K$. Equivalently there is
$G_K$-invariant pairing $$T\times \Tw(T)\map{}R(1)$$
which is symmetric when
the underlying group of $\Tw(T)$ is identified with that of $T$.
We assume that the local condition $\sel$ is its 
own exact orthogonal complement
under the induced local pairing
$$\langle \ ,\ \rangle _v:H^1(K_v,T)\times H^1(K_{\bar{v}},T)
\map{}R$$ for every place $v$ of $K$.

\item
\begin{enumerate}
\item The action of $G_K$ on $\Tbar$ extends to an action of $G_\Q$ and
the action of $\tau$ splits $\Tbar= \Tbar^+\oplus
\Tbar^-$ into one-dimensional eigenspaces,
\item The condition $\sel$ propagated to $\Tbar$ is stable under the
action of $G_\Q$,
\item If H.4 is assumed to hold then the residual
 pairing $$\Tbar\times \Tbar\map{}(R/\gm)(1)$$ satisfies
$(s^\tau,t^\tau)=(s,t)^\tau$ for all $s,t\in T$.
\end{enumerate}
\end{enumerate}
\renewcommand{\labelenumi}{(\alph{enumi})}

While Hypotheses H.0--H.3 are similar to hypotheses used in 
\cite{mazur-rubin}, Hypothesis H.4, the self-duality of $T$ (up to a twist), 
is not used by those authors, but plays an essential role here.  
Hypothesis H.5 is made to overcome a technical difficulty: in the applications
to Iwasawa theory, we will want to deal with $T=T_p(E)\otimes\Lambda$, where
$E_{/\Q}$ is an elliptic curve and $\Lambda$ is the Iwasawa algebra associated
to the anti-cyclotomic $\Z_p$-extension of $K$.  The natural action of $G_K$
on $T_p(E)\otimes\Lambda$ does not extend naturally to an action of $G_\Q$, 
but the action on the residual representation does.

We remark that the choice of $\pl$ plays no role in any of the hypotheses.
Hypothesis H.3  implies that the local condition $\sel$
is cartesian on $\Quot(T)$ at every place of $K$ by Lemma
\ref{unramified cartesian}.
When hypothesis H.4 holds, it can be shown the
local  pairing $$H^1(K_\lambda,T)\times H^1(K_\lambda,T)\map{}R$$
at any degree two prime $\lambda$ of $K$ is symmetric.

\begin{Rem}\label{base change}
It is easily seen that hyotheses H.0--H.5 are stable under base change in
the obvious sense.  See Remark \ref{functorality}.
\end{Rem}

\begin{Rem}\label{pairings}
The reader who is puzzled by the pairing of H.4 would do
well to keep the following example in mind.  If $R=\Z_p$, $T$ is the
$p$-adic Tate module of  an elliptic curve over $\Q$, and
$e:T\times T\map{}\Z_p(1)$ is the Weil pairing, then the pairing
$(s,t)=e(s,t^\tau)$ has the desired
properties.  The function $t\mapsto t^\tau$ defines a $G_{K_v}$-module
isomorphism $\Tw(T)\map{}T$ such that the composition of isomorphisms
$$H^1(K_{\bar{v}},T)\map{} H^1(K_v,\Tw(T))\map{}H^1(K_v,T)$$
is the usual action of complex conjugation.  Using this identification
the local pairing of H.4 is exactly the usual local Tate
pairing.

More generally, whenever the action of $G_K$ on $T$ extends to an
action of $G_\Q$, the existence of a pairing of the type
described in H.4
is equivalent to the existence of a skew-symmetric,
Galois-equivariant pairing on $T$.
As noted above, in the applications to Iwasawa theory we will
want to deal with modules for which the action does not extend.
\end{Rem}

\begin{Lem}\label{H.5 application}
Suppose $R$ is principal and Artinian of length $k$, and that
H.1 and H.3 hold.
If $0\le i\le k$  and $\pi$ is a generator of $\gm$, then the maps
$$T/\gm^i T\map{\pi^{k-i}}T[\gm^i]\map{}T$$ induce isomorphisms
$$H^1_\sel(K,T/\gm^iT)\map{}H^1_\sel(K,T[\gm^i])
\map{}H^1_\sel(K,T)[\gm^i].$$
\end{Lem}
\begin{proof}
See Remark \ref{cartesian identifications}, and Lemma 3.5.4 of
\cite{mazur-rubin}.
\end{proof}


\subsection{The Cassels-Tate pairing}\label{generalized cassels}


In this subsection we construct a generalized 
form of the Cassels-Tate pairing.  
Our exposition closely follows that of \cite{flach}.  See also \cite{guo}
and \cite{milne}.

Let $R$ be a principal Artinian coefficient ring  of length $k$
 and $T$ an object of $\Mod_{R,G_K}$.  Fix a generator $\pi$ of
the maximal ideal  $\gm$ of  $R$. Let $T^*=\Hom(T, R(1))$ and
fix a Selmer structure $\sel$ on $T$.
Let $\sel^*$ denote the dual Selmer structure on $T^*$.
In all that follows we assume that  $(T,\sel)$ and $(T^*,\sel^*)$
satisfy hypotheses H.0--H.5

At every place $v$ of $K$ set
$$H^1_\msel(K_v,T)=H^1(K_v,T)/H^1_\sel(K_v,T)$$ and similarly for $T^*$.
Hypothesis H.3 implies that
for any positive integers $s$ and $t$ with $s+t\le k$, and any
place $v$ of $K$, there are exact sequences
\begin{eqnarray}
\label{first}0\map{}H^1_\msel(K_v, T/\gm^tT)\map{\xi}&
H^1_\msel(K_v,T/\gm^{s+t}T)&\map{}H^1_\msel(K_v,T/\gm^sT)\end{eqnarray}
\begin{eqnarray}\label{second}H^1_{\sel^*}(K_v, T^*[\gm^{s}])\map{}&
H^1_{\sel^*}(K_v,T^*[\gm^{s+t}])&\map{\xi}H^1_{\sel^*}(K_v,T^*[\gm^t])\map{}0
\end{eqnarray}
where the arrows labeled $\xi$ are induced by $\pi^s:T\map{}T$.

 We want to construct a pairing
$$H^1_\sel(K,T/\gm^sT)\times H^1_{\sel^*}(K,T^*[\gm^t])
\map{} R$$ for any positive integers $s$ and
 $t$ with $s+t\le k$.
Suppose we are given classes in $H^1_\sel(K,T/\gm^sT)$ and
$H^1_{\sel^*}(K,T^*[\gm^t])$ represented by cocycles
$$a\in Z^1(K,T/\gm^sT)\hspace{1cm}b\in Z^1(K,T^*[\gm^t]).$$
We will repeatedly use the fact that for any topological group $G$
the continous cochain funtor
$C^i(G,\ )$ from $R$-modules to $R$-modules is exact,
and  so in particular we have surjective maps
$$C^1(K,T/\gm^{s+t}T)\map{}C^1(K,T/\gm^sT)
\hspace{1cm}C^1(K,T^*[\gm^{s+t}])\map{\pi^s}C^1(K,T^*[\gm^t])$$

Choose cochains $\alpha\in C^1(K,T/\gm^{s+t}T)$ and $\beta\in C^1(K,T^*[
\gm^{s+t}])$ which map to $a$ and $b$ respectively.  Let $d$ be the coboundary
operator. From $\pi^s d\beta=db$ it follows that
$d\beta$ is killed by $\pi^s$, and
similarly $d\alpha$ reducing to zero in $C^2(K,T/\gm^sT)$ implies that
$d\alpha$ is divisible by $\pi^s$ in $C^2(K,T/\gm^{s+t}T)$.
Therefore $d\alpha\cup d\beta=0$ and
$$d(d\alpha\cup \beta)=d^2\alpha\cup\beta+d\alpha\cup d\beta=0$$
so that $d\alpha\cup\beta$ lives in $Z^3(K, R(1))$
(we view the cup product as taking values in $ R(1)$-valued
cochains using the natural pairing $T\otimes T^*\map{} R(1)$).
By Theorem I.4.10 of \cite{milne}, $H^3(K, R(1))=0$, and  so there is an
$\epsilon\in C^2(K, R(1))$ with $$d\epsilon=d\alpha\cup\beta.$$

By the exact sequence (\ref{second})
there is a  $\beta'_v\in Z^1_{\sel^*}(K_v,T^*[\gm^{s+t}])$
 such that $\pi^s\beta'_v=b_v$, where
$Z^1_{\sel^*}(K_v,T^*[\gm^{s+t}])\subset Z^1(K_v,T^*[\gm^{s+t}])$
is the preimage of $H^1_{\sel^*}(K_v,T^*[\gm^{t}])$
under multiplication by $\pi^s$. The cochain
$\alpha_v\cup\beta'_v-\epsilon_v\in C^2(K_v, R(1))$
is in fact a coboundary, and we define the pairing
\begin{equation}
(a,b)_{s,t}=\sum_v\inv_v(\alpha_v\cup\beta'_v-\epsilon_v).
\end{equation}
It can be checked that this is independent of all choices made.

\begin{Prop}\label{cassels}
For positive integers $s$ and $t$ with $s+t\le k$ there is a pairing
$$(\ ,\ )_{s,t}:H^1_\sel(K,T/\gm^sT)\times H^1_{\sel^*}(K,T^*[\gm^t])
\map{} R$$ whose kernels on the left and right are the images
of \begin{eqnarray*}
H^1_\sel(K,T/\gm^{s+t}T)&\map{\hspace{.3cm}}&H^1_\sel(K,T/\gm^sT)\\
H^1_{\sel^*}(K,T^*[\gm^{s+t}])&\map{\pi^s }&H^1_{\sel^*}(K,T^*[\gm^t]).
\end{eqnarray*}
\end{Prop}
\begin{proof}
The construction of the pairing is above.  The computation of the kernels
is a straightforward modification of the methods of \cite{flach}.
\end{proof}

\begin{Thm}\label{structure}  There is an
$R$-module $M$ and an integer $\epsilon$ such that
$$H^1_\sel(K,T)\iso R^\epsilon\oplus M\oplus M.$$ By the structure theorem
for finitely-generated modules over $R$, we may assume $\epsilon\in\{0,1\}$.
\end{Thm}
\begin{proof}
Abbreviate $\cH= H^1_\sel(K,T)$, and for $1\le s< k$ define
$$V_s=\cH[\gm^s]/\gm \cH[\gm^{s+1}] \hspace{1cm}
W_s =\cH[\gm]/\gm^s \cH[\gm^{s+1}].$$
We claim that for $0\le s<k$, the $R/\gm$-vector space $V_s$
is even dimensional.  The claim then follows easily from this
and the structure theorem for finitely-generated $R$-modules.

There is an exact sequence
$$0\map{}V_{s-1}\map{}V_s\map{\pi^{s-1}}W_s.$$
Using hypothesis H.4 and
Lemma \ref{H.5 application},  we may identify $$H^1_{\sel^*}(K,T^*[\gm])
\iso H^1_\sel(K,T[\gm])\iso\cH[\gm]$$ and $H^1_\sel(K,T/\gm^sT)\iso
\cH[\gm^s]$. Proposition \ref{cassels} therefore gives a nondegenerate
pairing of $R/\gm$-vector spaces
$$(\ ,\ )_{s,1}:V_s\times W_s\iso \cH[\gm^s]/\gm\cH[\gm^{s+1}]
\times \cH[\gm]/\gm^s\cH[\gm^{s+1}]\map{}R[m].$$

We define a pairing 
$$\langle\ ,\ \rangle: V_s\times V_s\map{}R[m]$$ 
by $\langle a,b\rangle=(a,\pi^{s-1}b)_{s,1}$.  
The kernel on the right is $V_{s-1}$.
If we can show that this pairing is alternating, then $V_s/V_{s-1}$
is even dimensional for every $1\le s<k$, and the claim follows.
To check that this is alternating we must
verify $$(a,\pi^{s-1}b)_{s,1}=-(b,\pi^{s-1}a)_{s,1}.$$
We denote by $\phi:T\map{}\Tw(T)$ the identity map on underlying groups
and by $\psi$ the change of group isomorphisms
$$(G_K,T)\map{}(G_K,\Tw(T))\hspace{1cm}(G_{K_v},T)
\map{}(G_{K_{\bar{v}}},\Tw(T))$$ of Section \ref{hypotheses}.  We also
denote by $\psi$ the induced map on cochains and cohomology.
Fix $\alpha$ and $\beta$ in $C^1(F,T[\gm^{s+1}])$ with $\pi\alpha=a$
and $\pi\beta=b$,  
and choose $\epsilon_1$ and $\epsilon_2$ in $C^2(F,R(1))$
satisfying $$d\alpha\cup\psi(\beta)=d\epsilon_1\hspace{2cm}
d\beta\cup\psi(\alpha)=d\epsilon_2$$ and for every place
$v$ of $F$ elements $\alpha'_v$ and $\beta'_v$ in $H^1_\sel(F_v,T[\gm^{s+1}])$
which map to $a_v$ and $b_v$ under multiplication by $\pi$.
Then \begin{eqnarray*}
(a,\pi^{s-1}b)_{s,1}&=&\sum_v\inv_v(\alpha_v\cup\psi(\beta'_{\bar{v}})-
\epsilon_{1,v})\\
(b,\pi^{s-1}a)_{s,1}&=&\sum_v\inv_v(\beta_v\cup\psi(\alpha'_{\bar{v}})-
\epsilon_{2,v})\end{eqnarray*} where unprimed cochains are localizations
of global cochains, and primed cochains are (typically) not.  
Both $\alpha_v-\alpha'_v$ and $\beta_v-\beta'_v$ lie in $C^1(F_v,T[\gm])$,
and so $$(\alpha_v-\alpha'_v)\cup\psi(\beta_{\bar{v}}-\beta'_{\bar{v}})=0$$
which implies \begin{equation}\label{alternating equation}
\alpha_v\cup\psi(\beta_{\bar{v}})+\alpha'_v\cup\psi(\beta'_{\bar{v}})
=\alpha_v\cup\psi(\beta'_{\bar{v}})+\alpha'_v\cup\psi(\beta_{\bar{v}}).
\end{equation}

Given a topological group $G$, if 
$\R^*$ is the standard resolution of $\Z$ by projective $G$-modules
then one can form the tensor square resolution $\R^*\otimes \R^*$.  For
a topological $G$-module $M$ denote by $CC^*(G,M)$ the cochain
complex $\Hom(\R^*\otimes \R^*,M)$ of continuous homomorphisms.  
The cohomology
of $CC^*$ agrees with the usual continuous cohomology (see \cite{flach})
and the automorphism $\rho$ of $CC^*$ induced by the 
automorphism $r_1\otimes r_2
\mapsto r_2\otimes r_1$ of $\R^*\otimes \R^*$ induces the 
identity on cohomology.
It follows from the results of V.3.6 of \cite{brown} that there is a 
commutative diagram of complexes
$$\xymatrix{
C^*(K_v,T)\otimes C^*(K_v,\Tw(T))\ar[r]^\cup\ar[d]^s& 
CC^*(K_v,T\otimes\Tw(T))\ar[r]\ar[d]^{(\rho,\mathrm{tr})}&
CC^*(K_v,R(1))\ar[d]^\rho\\
C^*(K_v,\Tw(T))\otimes C^*(K_v,T) \ar[r]^\cup\ar[d]^\psi& 
CC^*(K_v,\Tw(T)\otimes T)\ar[r]\ar[d]^\psi&CC^*(K_v,R(1))\ar[d]^{-\tau}\\
C^*(K_{\bar{v}},T)\otimes C^1(K_{\bar{v}},\Tw(T))\ar[r]^\cup& 
CC^*(K_{\bar{v}}, T\otimes\Tw(T))\ar[r]&CC^*(K_{\bar{v}},R(1))}$$
in which $\mathrm{tr}:T\otimes\Tw(T)\map{}\Tw(T)\otimes T$ takes 
$t_1\otimes t_2$ to $t_2\otimes t_1$, $s$ is the map 
$$a\otimes b\map{}(-1)^{\deg(a)\deg(b)}b\otimes a,$$
and $\tau$ is the change of group $(G_{K_v},R(1))\map{}(G_{K_{\bar{v}}},
R(1))$ which is conjugation by $\tau$ on the groups and action by $\tau$
on $R(1)$.  Commutativity of the bottom right square follows from the 
symmetry $(t_1,\phi(t_2))=(t_2,\phi(t_1))$ of the pairing of 
H.4.  The upshot of the diagram is the relation
\begin{equation}\label{flippity flippity}
x\cup\psi(y)=(-1)^{\deg(x)\deg(y)+1}(y\cup\psi(x))^\tau\end{equation} 
where $x$ and $y$ are in
$C^*(K_v,T)$ and $C^*(K_{\bar{v}},T)$, respectively.
There is a similar global diagram obtained by ignoring all $v$'s 
and $\bar{v}$'s, and the relation (\ref{flippity flippity}) holds 
for $x, y\in C^*(K,T)$.

From  (\ref{alternating equation}) we now deduce
\begin{eqnarray}\label{alternating equation 2}\lefteqn{
\big(\alpha\cup\psi(\beta)-\epsilon_1-(\epsilon_2)^\tau\big)_v
+\alpha'_v\cup\psi(\beta'_{\bar{v}})=}\hspace{2cm}\\ \nonumber
& &\alpha_v\cup\psi(\beta'_{\bar{v}})-\epsilon_{1,v}+
\alpha'_v\cup\psi(\beta_{\bar{v}})-(\epsilon_{2,\bar{v}})^\tau.
\end{eqnarray}
It follows from (\ref{flippity flippity}) and the definition of $\epsilon_i$
that $\alpha\cup\psi(\beta)-\epsilon_1+(\epsilon_2)^\tau$ is a 
$2$-cocycle, and so by the reciprocity law of class field theory
the sum of its local inraviants is zero.  The local invariant
of $\alpha'_v\cup\psi(\beta'_{\bar{v}})$ is zero by the assumption that
$\sel$ is everywhere self-orthogonal under the local pairing.  Again using
(\ref{flippity flippity}) we obtain
$$\sum_v\inv_v(\alpha_v\cup\psi(\beta'_{\bar{v}})-\epsilon_{1,v})
=-\sum_v\inv_v((\beta_{\bar{v}}\cup\psi(\alpha'_v)-\epsilon_{2,\bar{v}})^\tau)
$$ and the claim now follows from Galois invariance of the local
invariant map.
\end{proof}


\subsection{Modules over principal Artinian rings}
\label{Artinian}


Throughout Subsection \ref{Artinian} we fix a coefficient ring $R$
which is assumed to be principal and Artinian of length $k$. Let
$(T,\sel,\pl)$ be a Selmer triple satisfying hypotheses H.0--H.5.
We assume that $\pl\subset\pl_k(T)$, so that $I_n R=0$ for every
$n \in\pn=\pn(\pl)$.  By H.0 and Proposition
\ref{locally free}, this implies that the local conditions
$H^1_\f(K_\lambda,T)$ and $H^1_\tr(K_\lambda,T)$ are free rank two
$R$-modules.

Set $\Tbar=T/\gm T$, and abbreviate
$$\cH^a_b(c)=H^1_{\sel^a_b(c)}(K,T)\hspace{1cm}
\bar{\cH}^a_b(c)=H^1_{\sel^a_b(c)}(K,\Tbar)$$ for $abc\in\pn=\pn(\pl)$.
For any $c\in H^1(K,T)$ and any place $v$ of $K$ we denote by $c_v$ the
image of $c$ in $H^1(K_v,T)$ and by $\langle \ ,\ \rangle _v$ the local Tate
pairing $$H^1(K_v,T)\times H^1(K_{\bar{v}},T)\map{}R$$ of H.4.
For any integer $n$, $\nu(n)$
denotes the number of prime divisors of $n$.
Recall that $\tau\in \Gal(\bar{\Q}/\Q)$ is a fixed complex conjugation.
If $M$ is any $R/\gm$-vector space on which $\tau$ acts we
denote by $M^+$ and $M^-$ the
subspaces on which $\tau$ acts by $+1$ and $-1$ respectively.

\begin{Lem}\label{transverse cartesian}
The Selmer triple $(T,\sel(n),\pl(n))$ satisfies
H.0--H.5 for any $n\in\pn$.
\end{Lem}
\begin{proof} See Lemma 3.7.4 of \cite{mazur-rubin} for the case of
H.3.  The other cases are trivial.
\end{proof}

\begin{Def}
For any $n\in\pn$ we let $\rho(n)^\pm$ be the $R/\gm$-dimension of
$\bar{\cH}(n)^\pm$, and set $\rho(n)=\rho(n)^++\rho(n)^-$.
\end{Def}

\begin{Lem}\label{parity} For any  $n\ell\in\pn$  \begin{enumerate}
\item if $\loc_\ell \big(\bar{\cH}(n)^\pm\big)\not=0$ then
$\rho(n\ell)^\pm=\rho(n)^\pm-1$ and
$\loc_\ell\big(\bar{\cH}(n\ell)^\pm\big)=0,$
\item  if $\loc_\ell \big(\bar{\cH}(n)^\pm\big)=0$ then
$\rho(n\ell)^\pm=\rho(n)^\pm+1$.
\end{enumerate}

In particular this implies that $\rho(n)\pmod{2}$ is
independent of $n\in\pn$.
\end{Lem}
\begin{proof} Assume that $\loc_\ell
\big(H^1_{\sel(n)}(K,\Tbar)^\pm\big)\not=0$ and consider the exact
sequences
\begin{eqnarray}\label{parity sequence}
& &0\map{}H^1_{\sel_\ell(n)}(K,\Tbar)\map{}H^1_{\sel(n)}(K,\Tbar)
\map{}H^1_\f(K_\ell,\Tbar)\\&
&0\map{}H^1_{\sel(n)}(K,\Tbar)\map{}H^1_{\sel^\ell(n)}(K,\Tbar)\map{}
H^1_\s(K_\ell,\Tbar).\nonumber\end{eqnarray} By global duality
(Theorem \ref{global duality}) the images of the rightmost arrows
are exact orthogonal complements under the $G_\Q$-invariant local
Tate pairing.  Furthermore the action of complex conjugation
splits $H^1_\f(K_\ell,\Tbar)$ and  $H^1_\s(K_\ell,\Tbar)$ each
into one-dimensional eigenspaces by H.5 and
the isomorphisms
$$H^1_\f(K_\ell,\Tbar)\iso \Tbar\iso H^1_\s(K_\ell,\Tbar)\otimes\bk^\times$$
of Proposition \ref{locally free}. It follows that
$H^1_{\sel(n)}(K,\Tbar)^\pm= H^1_{\sel^\ell(n)}(K,\Tbar)^\pm$ and
therefore $H^1_{\sel_\ell(n)}(K,\Tbar)^\pm= H^1_{\sel(\ell
n)}(K,\Tbar)^\pm$. This proves (a).

Assume that $\loc_\ell \big(H^1_{\sel(n)}(K,\Tbar)^\pm\big)=0$.
Again applying global duality to the exact sequences (\ref{parity
sequence}) we see that it suffices to show
$H^1_{\sel^\ell(n)}(K,\Tbar)^\pm= H^1_{\sel(n\ell)}(K,\Tbar)^\pm$.
If $c\in H^1_{\sel^\ell(n)}(K,\Tbar)^\pm$ then the local image of
$c$ at $\ell$ is self-orthogonal under the local pairing. Indeed,
the reciprocity law of class field theory and the isotropy of the
local conditions $\sel(n)$ (by H.4) imply
$$\langle c_\ell,c_\ell\rangle _\ell=
\sum_v\langle c_v,c_{\bar{v}}\rangle _v=0$$ where the sum
is over all places of $K$.  Therefore the localization of
$H^1_{\sel^\ell(n)}(K,\Tbar)^\pm$ at $\ell$ is a maximal isotropic
subspace of $H^1(K_\ell,\Tbar)^\pm$ and an elementary linear
algebra exercise shows that the only two such subspaces are
$H^1_\f(K_\ell,\Tbar)^\pm$ and $H^1_\tr(K_\ell,\Tbar)^\pm$.
Therefore $H^1_{\sel^\ell(n)}(K,\Tbar)^\pm$ is equal to either
$H^1_{\sel(n)}(K,\Tbar)^\pm$ or $H^1_{\sel(n\ell)}(K,\Tbar)^\pm$.
Returning to the exact sequences (\ref{parity sequence}) we see
that the first possibility contradicts the assumption $\loc_\ell
\big(H^1_{\sel(n)}(K,\Tbar)^\pm\big)=0$.
\end{proof}

By Theorem \ref{structure} and Lemma \ref{transverse cartesian},
for each $n\in\pn$ there is an $R$-module $M(n)$ and an integer
$\epsilon$ such that
\begin{equation}\label{structure decomposition}
\cH(n)\iso R^\epsilon\oplus M(n)\oplus M(n).\end{equation}
 By the structure theorem
for finitely-generated modules over $R$, we can (and do) take
$\epsilon\in\{0,1\}$.
It will be seen momentarily that $\epsilon$ is independent of $n$.

\begin{Def}
For $n\in\pn$ and with notation as in the preceeding theorem we define
\begin{enumerate}
\item $\lambda(n)=\len(M(n))$,
\item the \emph{stub Selmer module} $\stub(n)=\gm^{\lambda(n)}\cH(n)$.
\end{enumerate}
\end{Def}
The reader is invited to compare the above  definitions
with Definitions 4.1.2 and 4.3.1 of \cite{mazur-rubin}.

\begin{Prop}\label{epsilon}
The integer $\epsilon$ appearing in the decomposition
(\ref{structure decomposition}) is congruent
to $\rho(n)\pmod{2}$ and is therefore independent of $n\in\pn$
by Lemma \ref{parity}.
\end{Prop}
\begin{proof} We have $$\epsilon+2\dim_{R/\gm} M(n)[\gm]=\dim_{R/\gm}
 \cH(n)[\gm]=\rho(n),$$
the second equality by Lemma \ref{H.5 application}.
\end{proof}

\begin{Lem} For $mn\in\pn$,  the image of
$\cH^m(n)$ in $\bigoplus_{\lambda|m}H^1(K_\lambda,T)$ is maximal isotropic
under the sum of the local Tate pairings.
\end{Lem}
\begin{proof}
Let $A$ be the image of $\cH^m(n)$ in $\bigoplus_{\lambda|m}H^1(K_\lambda,T)$.
The local condition $\sel^m(n)$ is maximal isotropic away from $m$
under the local Tate pairing, and the reciprocity law of class field theory
implies that for any $c,d\in H^1_{\sel^m(n)}(K,T)$
$$\sum_{\lambda|m}\langle c_\lambda,
d_\lambda\rangle _\lambda=\sum_{\mathrm{all\ }v}
\langle c_v,d_{\bar{v}}\rangle _v=0$$
which shows that $A\subset A^\perp$. By global duality
(Theorem \ref{global duality})
\begin{eqnarray*}\len(A)&=&\len(\cH^m(n)/\cH(n))+\len(\cH(n)/\cH_m(n))\\
&=&2k\cdot\nu(n).
\end{eqnarray*}
The sum of the lengths of $A$ and $A^\perp$ must be $4k\cdot\nu(m)$
and we conclude that $\len(A)=\len(A^\perp)$ and so $A=A^\perp$.
\end{proof}

\begin{Lem}
For some $\delta\ge 0$, 
$\cH^\ell(n)/(\cH(n)+\cH(\ell n))\iso (R/\gm^\delta)^2$.
\end{Lem}
\begin{proof}
We first construct a non-degenerate, alternating, $R$-bilinear,
$R$-valued pairing on the module $\cH^\ell(n)/(\cH(n)+\cH(\ell n))$.
Let $A$ be the local image of $\cH^\ell(n)$ in $H^1(K_\ell,T)$.
$A$ is maximal isotropic by the previous lemma.
Write $A_\f$ and $A_\tr$ for the intersections
of $A$ with $H^1_\f(K_\ell,T)$ and $H^1_\tr(K_\ell,T)$, respectively.
Localization at $\ell$ gives an isomorphism $$\cH^\ell(n)/(\cH(n)+\cH(\ell n))
\iso A/(A_\f+A_\tr)$$ and
it is on this $R$-module that we define the pairing.

If $x\in A$ write $x_\f$ and $x_\tr$ for the projections of $x$ onto the
finite and transverse submodules.  For $x,y\in A$ we define the symbol
$[x,y]\in R$ by $[x,y]=\langle x_\f,y_\tr\rangle $.  That $[x,y]=-[y,x]$
follows immediately from $\langle x,y\rangle =0$
and the isotropy of the finite and
transverse submodules.  Suppose $x\in A$ is in the kernel of this pairing,
then $0=\langle x_\f,y_\tr\rangle =\langle 
x_f,y\rangle $ for every $y\in A$ and so $x_\f\in A$
by maximal isotropy of $A$.  It follows that $x_\tr\in A$
and so $x\in A_\f+A_\tr$, proving that the pairing is non-degenerate.

We now have that $$\cH^\ell(n)/(\cH(n)+\cH(\ell n))\iso D\oplus D$$ for
some $R$-module $D$.  Since $\cH^\ell(n)/\cH(n)$ injects into
$H^1_\s(K_\ell,T)$ which is free of rank 2, it follows that
$\cH^\ell(n)/(\cH(n)+\cH(\ell n))$ can be generated by two elements.  
Therefore $D$ is cyclic.
\end{proof}

\begin{Lem}\label{pretty diagram}
There are $a$, $b$, and $\delta$ greater than or equal to zero such that
in the following diagram the cokernel of each inclusion is a direct sum
of two cyclic $R$-modules of the indicated lengths.

\begin{center}
\setlength{\unitlength}{1cm}
\begin{picture}(4,3.5)
\put(1.9,3){$\cH^\ell(n)$}

\put(1,2){\vector(1,1){.75}} \put(-.6,2.4){$k-a,\ k-b$}
\put(3.6,2){\vector(-1,1){.75}} \put(3.6,2.4){$a+\delta,\ b+\delta$}

\put(.4,1.5){$\cH(n)$}
\put(3.4,1.5){$\cH(n\ell)$}

\put(1.9,.5){\vector(-1,1){.75}}\put(.75,.6){$a,\ b$}
\put(2.7,.5){\vector(1,1){.75}}\put(3.3,.6){$k-a-\delta,\ k-b-\delta$}

\put(1.9,0){$\cH_\ell(n)$}
\end{picture}
\end{center}
\end{Lem}
\begin{proof} The relation between the lower left and upper 
left quotients follows
from global duality, and similarly for the lower and upper 
right quotients.
The relation between lower left and
upper right quotients, and also the relation between lower right and
upper left, follows from the preceeding lemma.
\end{proof}

\begin{Prop}\label{induction} For $n\ell\in\pn$
 $$\loc_\ell(\stub(n))=0\implies \loc_\ell(\stub(\ell n))=0.$$
\end{Prop}
\begin{proof}  Keeping the notation as in the diagram of 
Lemma \ref{pretty diagram},
 $\loc_\ell (\stub(n))=0$ implies that $\gm^{\lambda(n)}$ kills
the lower left quotient, and so $a, b\le\lambda(n)$.  The diagram immediately
implies \begin{eqnarray*}\lambda(n\ell)&=&\lambda(n)+k-a-b-\delta\\&\ge&
k-a-\delta,\ k-b-\delta\end{eqnarray*}
so that $\gm^{\lambda(n\ell)}$ kills the lower right quotient.
The claim follows.
\end{proof}


\subsection{Bounding the Selmer group}
\label{bound}


Throughout this subsection $R$ is a fixed discrete valuation ring
with uniformizing parameter $\pi$.  Let $(T,\sel,\pl)$ be a Selmer
triple satisfying Hypotheses H.0--H.5, and suppose
$\pl_s(T)\subset\pl$ for $s\gg 0$.  If $\Phi$ denotes the field of
fractions of $R$, $\D=\Phi/R$, and $A=T\otimes_R \D$, then we
obtain a Selmer structure on $A$, still denoted $\sel$, by
propagating $\sel\otimes\Phi$ from $T\otimes \Phi$ to $A$. The
following theorem is the technical core of this paper.

\begin{Thm}\label{dvr bound}
Suppose there is a Kolyvagin system $\kappa\in \KS(T,\sel,\pl)$ with
$\kappa_1\not=0$.  Then $H^1_\sel(K,T)$ is  a free rank-one $R$ module,
and there is a finite $R$-module $M$ such that
$$H^1_\sel(K,A)\iso \D\oplus M\oplus M.$$  Furthermore
$\len_R(M)\le \len_R( H^1_\sel(K,T)/R\cdot\kappa_1)$.
\end{Thm}

We will prove this through a series of lemmas. For any $k\ge 0$ we
define $$R^{(k)}=R/\gm^k \hspace{1cm} T^{(k)}=T/\gm^kT
\hspace{1cm}\pl^{(k)}=\pl\cap\pl_k(T).$$ By Remark \ref{base
change}, the Selmer triple $(T^{(k)},\sel,\pl^{(k)})$ satisfies
hypotheses H.0--H.5, and we may invoke the definitions and results
of the preceeding section.  In particular for
$\ell\in\pn^{(k)}=\pn(\pl^{(k)})$ we have a decomposition
$$H^1_{\sel(n)}(K,T^{(k)})\iso R^{(k),\epsilon}\oplus M^{(k)}(n)
\oplus M^{(k)}(n)$$ in which $\epsilon\in\{0,1\}$ is independent of both
$n$ and $k$ (by Lemma \ref{parity}).  We define
$$\lambda^{(k)}(n)=\len_R(M^{(k)}(n))\hspace{1cm}
\stub^{(k)}(n)=\gm^{\lambda^{(k)}(n)}H^1_{\sel(n)}(K,T^{(k)}).$$
We obtain, by Remark \ref{functorality}, a Kolyvagin system
$\kappa^{(k)}\in\KS(T^{(k)},\sel,\pl^{(k)})$.

\begin{Lem}\label{Cheb} Suppose  we are given elements
$$c^+\in H^1(K,\Tbar)^+ \hspace{1cm} c^-\in H^1(K,\Tbar)^-.$$
There are infinitely many primes $\lambda\in\pl^{(2k-1)}$ such
that $c^\pm\not=0\implies \loc_\lambda(c^\pm)\not=0$.
\end{Lem}
\begin{proof} We consider only the case where $c^+$, $c^-$ are both nonzero, 
the other case being entirely similar.
Let $F/\Q$ be the extension of Hypothesis H.2,
and let $L$ be the Galois closure (over $\Q$) of 
$K(T^{(2k-1)},\mu_{p^{2k-1}})$.  Since $F/\Q$ is Galois by hypothesis,
$L\subset F(\mu_{p^\infty})$, and so restriction
$$H^1(K,\Tbar)\map{}H^1(L,\Tbar)^{\Gal(L/K)}\iso
\Hom(G_L,\Tbar)^{\Gal(L/K)}$$ is an injection.
We identify $c^\pm$ with its image under restriction.  
Let $E$ be the smallest extension of $L$ with $c^\pm(G_E)=0$, and
set $G=\Gal(E/L)$.  Then $G$ is an $\F_p$-vector space with a natural
action of $\Gal(L/\Q)$, and we let $G^\pm$ be the $\pm$-eigenspace for
the action of $\tau$.

We claim that the maps
\begin{equation}\label{cheb maps}
c^+:G^+\map{}\Tbar^+\hspace{1cm}c^-:G^+\map{}\Tbar^-
\end{equation} are nontrivial.  Indeed, if $c^+(G^+)=0$ then
$c^+(G)=c^+(G^-)\subset \Tbar^-$, and so $R\cdot c^+(G)$ is an
$R[G_K]$-submodule of $\Tbar$ contained in $\Tbar^-$.  This contradicts
Hypotheses H.1 and H.5 (a).
Similar considerations apply to $c^-$.

The kernels of the maps (\ref{cheb maps}) have codimension $\ge 1$,
and so there is an $\eta\in G^+$ for which $c^\pm(\eta)$ are both
nonzero, and we may choose some $\sigma\in G$ such that 
$\eta=(\tau\sigma)^2$.  By the Cebotarev theorem, there are infinitely many
primes $\ell$ of $\Q$ whose Frobenius class
in $\Gal(E/\Q)$ is equal to $\tau\sigma$, and at which the localizations
of $c^\pm$ are unramified.  For such an $\ell$, the image of $c^\pm$ under
$$H^1(K,\Tbar)\map{}H^1(K_\ell,\Tbar)\map{}H^1_\unr(K_\ell,\Tbar)\iso \Tbar$$
(the final isomorphism being evaluation at the Frobenius of the prime of $K$
above $\ell$) is equal to $\phi(c^\pm)\not=0$.
\end{proof}

\begin{Lem}\label{liftability}
If $n\in\pn^{(2k-1)}$ and $\stub^{(k)}(n)\not=0$ then the image of
$$H^1_{\sel(n)}(K,T^{(2k-1)})\map{}H^1_{\sel(n)}(K,T^{(k)})$$
is a free, rank-one $R^{(k)}$-submodule.
\end{Lem}
\begin{proof} Under the identification $H^1_{\sel(n)}(K,T^{(k)})\iso
H^1_{\sel(n)}(K,T^{(2k-1)})[\gm^k]$ of Lemma \ref{H.5 application},
the above map is identified with
$$H^1_{\sel(n)}(K,T^{(2k-1)})\map{\pi^{k-1}}H^1_{\sel(n)}
(K,T^{(2k-1)})[\gm^k].$$  The hypothesis $\stub^{(k)}(n)\not=0$
implies that $\len_R(M^{(2k-1)})<k$ and that $\epsilon=1$, hence the
image is isomorphic as an $R$-module to $\gm^{k-1}R^{(2k-1)}\iso R^{(k)}$.
\end{proof}

\begin{Lem}\label{lemma of the stub}
If $n\in\pn^{(2k-1)}$ then $\kappa_n^{(k)}\in
\stub^{(k)}(n)\otimes G_n$.
\end{Lem}
\begin{proof} We argue by induction on both $k$ and $\rho^{(k)}(n)$.  
Let $k>0$ be the minimal
integer for which the claim is false (for some $n$), 
and fix a generator for the cyclic group $G_\ell$ for
every $\ell\in\pn^{(2k-1)}$ so that we may identify
$H^1_{\sel(n)}(K,T^{(k)})\otimes G_n\iso H^1_{\sel(n)}(K,T^{(k)})$.

First suppose $\stub^{(k)}(n)\not=0$, so that in particular we are in
the case $\epsilon=1$, and $\lambda^{(k)}(n)<k$.  Let $i=\lambda^{(k)}(n)$.
By minimality of $k$, $\kappa^{(i)}_n\in \stub^{(i)}(n)$.  By Lemma 
\ref{H.5 application} we have an isomorphism of $R$-modules
$M^{(i)}\iso M^{(k)}[\gm^i]=M^{(k)}$, so that $\lambda^{(i)}(n)=
\lambda^{(k)}(n)=i$. This implies that $\stub^{(i)}(n)=0$, and so 
$\kappa^{(i)}_n=0$.  Appealing again to Lemma \ref{H.5 application}, this
is equivalent to $\pi^{k-i}\kappa_n^{(k)}=0$.  Now by Lemma \ref{liftability},
$\kappa_n^{(k)}$ is divisible by $\pi^i$ in $H^1_{\sel(n)}(K,T^{(k)})$,
proving this special case.

Now keep $k$ fixed as above and suppose that 
$n\in\pl^{(2k-1)}$ gives a counterexample with $\rho(n)$ minimal.  
The above case shows that $\stub^{(k)}(n)=0$.
By Lemma \ref{H.5 application}, $\rho(n)=0$ or $1$ implies that 
$\stub^{(k)}(n)=H^1_{\sel(n)}(K,T^{(k)})$, and so we must have 
$\rho(n)>1$.  

Case i: $\rho(n)^+$ and 
$\rho(n)^-$ are both nonzero.  Using Lemma \ref{H.5 application} we 
identify $H^1_{\sel(n)}(K,T^{(k)})[\gm]\iso H^1_{\sel(n)}(K,\Tbar)$.
If $\kappa^{(k)}(n)\not=0$ then it has some nonzero multiple $d\in
H^1_{\sel(n)}(K,T^{(k)})[\gm]$.  This $d$ has nontrivial projection onto
one of the $\tau$-eigencomponents of $H^1_{\sel(n)}(K,\Tbar)$.  Assume
that $d^+\not=0$.  By Lemma \ref{Cheb} we may choose a prime $\ell
\in\pl^{(2k-1)}$ at which both $d^+$ and some element of $H^1_{\sel(n)}
(K,\Tbar)^-$ have nontrivial localization.  By Lemma \ref{parity},
$\rho(n\ell)=\rho(n)-2$, and so by induction $\kappa^{(k)}(n\ell)\in
\stub^{(k)}(n\ell)$.  By Proposition \ref{induction},
$\loc_\ell(\kappa^{(k)}(n\ell))=0$, but then 
the Kolyvagin system relations imply that 
$\loc_\ell(\kappa^{(k)}_n)=0$, contradicting the choice of $\ell$.

Case ii: one of $\rho(n)^\pm$ is equal to zero.  Suppose $\rho(n)^-=0$,
so that $\rho(n)^+>1$. 
If $\kappa^{(k)}(n)\not=0$ then choose a nonzero
multiple of $\kappa^{(k)}(n)$, $d\in H^1_{\sel(n)}(K,T^{(k)})[\gm]^+$,
and a prime  $\ell\in\pl^{(2k-1)}$ for which $\loc_\ell(d)\not=0$.
By Lemma \ref{parity}, $\rho(n\ell)^\pm$ are both nonzero and 
$\rho(n\ell)=\rho(n)$.  Thus, by Case i, $\kappa^{(k)}_{n\ell}\in
\stub^{(k)}(n\ell)$.
By Proposition \ref{induction}, $\loc_\ell(\stub^{(k)}(n\ell))=0$,
but the Kolyvagin system relations guarantee that 
$\loc_\ell(\kappa^{(k)}_{n\ell})\not=0$.  This is a contradiction.
\end{proof}

\begin{proof}[Proof of Theorem \ref{dvr bound}.] Since
$H^1_\sel(K,T)\iso \mil H^1_\sel(K,T^{(k)})$, we must have $\kappa_1^{(k)}$ 
nonzero for $k\gg 0$.  Fix such a $k$.
Taking $n=1$ in Lemma \ref{lemma of the stub}, we have 
$\kappa_1^{(k)}\in\stub^{(k)}$, and in particular
$\stub^{(k)}\not=0$. Lemma \ref{H.5 application} implies that 
there are isomorphisms $$H^1_\sel(K,T^{(k)})\iso H^1_\sel(K,A[\gm^k])
\iso H^1_\sel(K, A)[\gm^k],$$
and we conclude that $$H^1_\sel(K,A)[\gm^k]\iso R/\gm^k\oplus M^{(k)}
\oplus M^{(k)}$$ with $\len_R(M^{(k)})< k$, and so
for some finite $R$-module $M\iso M^{(k)}$ there is an isomorphism
$H^1_\sel(K,A)\iso \D\oplus M\oplus M$.  

The compact Selmer group $H^1_\sel(K,T)$ is the $\pi$-adic Tate 
module of $H^1_\sel(K,A)$, and is therefore a free rank-one $R$-module.
Let $\lambda=\len_R(M)=\lambda^{(k)}(1)$.  By Lemma 
\ref{lemma of the stub},
$\kappa_1^{(k)}\in \gm^\lambda H^1_\sel(K,T^{(k)})$, and so
by the injectivity of 
$$H^1_\sel(K,T)/\gm^k H^1_\sel(K,T)\map{}H^1_\sel(K,T^{(k)})$$
(which is deduced from Lemma \ref{H.5 application}),
$\kappa_1\in\gm^\lambda H^1_\sel(K,T)$.  The claim follows.
\end{proof}

Let $E/\Q$ be an elliptic curve as in the statement of 
Theorem \ref{kolyvagin's theorem} of the introduction, and
let $\Sel_{p^\infty}(E/K)$ and $S_p(E/K)$ the $p$-power Selmer
groups defined there.
Define a Selmer structure $\sel$ on $V=T_p(E)\otimes\Q_p$
by taking the unramified local condition at each place $v$ of $K$
which does not divide $p$, and at $v|p$ take the
image of the local Kummer map
$$E(K_v)\otimes \Q_p\map{}H^1(K,V).$$ 
Define local conditions
on $T_p(E)$ and $E[p^\infty]\iso V/T_p(E)$ by
propagating $\sel$. By Proposition 1.6.8 of
\cite{rubin}, $H^1_\sel(K,E[p^\infty])=\Sel_{p^\infty}(E/K)$.

\begin{Thm}\label{kolyvagin's theorem II}(Kolyvagin)
Suppose there is an integer $s$ for which the Selmer triple
$(T_p(E),\sel,\pl_s)$ admits a Kolyvagin system with
$\kappa_1\not=0$.  Then $S_p(E/K)$ is free of rank one
over $\Z_p$ and there is a finite $\Z_p$-module
$M$ such that $$\Sel_{p^\infty}(E/K)\iso
(\Q_p/\Z_p)\oplus M\oplus M$$ with $\len_{\Z_p}(M)\le \len_{\Z_p}(
S_{p}(E/K)/\Z_p\cdot \kappa_1).$
\end{Thm}
\begin{proof} By Theorem \ref{dvr bound} we need only verify that Hypothesis
H.0--H.5 hold.  Hypothesis H.0 is trivial. 
Hypothesis H.1 follows from our assuption
that $G_K$ surjects onto $\Aut_{\Z_p}(T_p(E))$.
This assumption also implies that
$$H^1(K(E[p^\infty])/K, E[p])\iso H^1(GL_2(\Z_p),\F_p^2)=0$$
(for the second equality, apply the inflation-restriction
sequence to the subgroup $\mu_{p-1}\subset GL_2(\Z_p)$ embedded 
diagonally.)
Hence H.2 holds with $F=K(E[p^\infty])$.
The fact that $\sel$ is obtained by propagation from $V$
implies that the quotient of $H^1(K_v,T_p(E))$ by $H^1_\sel(K_v,T_p(E))$
is torsion-free for every place $v$, and hence
Hypothesis H.3 holds by Lemma 3.7.1 of \cite{mazur-rubin}.
The pairing of H.4 is the Weil pairing, modified as in 
Remark \ref{pairings}.  The orthogonality relations of that hypothesis
are equivalent to Tate local duality, by the same remark.
All of the conditions of H.5 hold for $T_p(E)$, 
hence also for $\bar{T}\iso E[p]$, 
using the fact that $E$ is defined over $\Q$.  The 
splitting of part (a) follows from the $\tau$-invariance of the
Weil pairing on $T_p(E)$; part (b) says that the images of the local Kummer
maps are stable under the $G_\Q$-action on (semi-) local cohomology;
part (c) follows from 
$$(s^\tau,t^\tau)=e(s^\tau,t)=e(s,t^\tau)^\tau=(s,t)^\tau,$$
where $e$ is the Weil pairing. 
\end{proof}

In the next section we will construct a Kolyvagin system from the 
Euler system of Heegner points.  Applying 
Theorem \ref{kolyvagin's theorem II} to this Kolyvagin system 
proves Theorem \ref{kolyvagin's theorem} 
of the introduction.


\subsection{Heegner points}
\label{Heegner points}


In this subsection we show that our theory is nonvacuous by
constructing a Kolyvagin system for $T=T_p(E)$ from the Heegner
point Euler system. Let $E_{/\Q}$ be an elliptic curve of
conductor $N$ and $K$ a quadratic imaginary field of discriminant
prime to $p$ and $\not=-3, -4$.  Assume that $p$ does not divide
$N$ and that all prime divisors of $N$ are split in $K$. Fix an
integral ideal $\ga$ of $\cO_K$ satisfying $\cO_K/\ga\iso \Z/N\Z$.
Let $\pl=\pl_1(T)$ and $\pn=\pn_1$.  For $\ell\in\pl$, we denote
by $a_\ell\in\Z$ the trace of the Frobenius at $\ell$ on $T_p(E)$.
The ideal $I_\ell\subset\Z_p$ is the smallest ideal containing
$\ell+1$ for which $\Frob_\lambda=\Frob_\ell^2$ acts trivially on 
$T/I_\ell T$, and hence on which $\Frob_\ell$ acts with
characteristic polynomial $X^2-1$.  Therefore $I_\ell$ is
generated by $a_\ell$ and $\ell+1$.

For every integer of the
form  $m=p^kn$ with $n\in\pn$ we let $h_m\in X_0(N)$ be the point
corresponding to the cyclic $N$-isogeny of complex tori
$$h_{m}=[\C/\cO_m\map{}\C/(\cO_m\cap\ga)^{-1}]$$
where $\cO_m$ is the order of conductor $m$ in $\cO_K$ and
$(\cO_m\cap\ga)^{-1}$ is the inverse of the invertible
$\cO_m$-ideal $(\cO_m\cap\ga)$.  The point $h_m$ is rational over the
ring class field of conductor $m$, which we denote by $K[m]$.
Let $J_0(N)$ be the Jacobian of $X_0(N)$, and embed $X_0(N)\hookrightarrow
J_0(N)$ by sending the cusp at $\infty$ to the origin.  The image of
$h_m$ in $J_0(N)$ is again denoted by $h_m$.  Fix a modular
parametrization $$J_0(N)\map{}E.$$  The image of $h_m$ is now denoted
by $P[m]\in E(K[m])$, the \emph{Heegner point of conductor $m$}.
If $n\ell\in\pn$ we have the Euler system relation (\cite{gross} Proposition
3.7, or \cite{pr87} Section 3.3, for example)
$$\Norm_{K[n\ell]/K[n]}P[n\ell]=a_\ell P[n]$$ and the
congruence \begin{equation}\label{congruence}
P[n\ell]\equiv \left(\frac{\lambda'}{K[n\ell]/\Q}\right)
 P[n]\pmod{\lambda'}\end{equation} where $\lambda'$ is any prime of
$K[n\ell]$ above $\ell$.

If $n\in\pn$ we set $\pg(n)=\Gal(K[n]/K)$ and
$G(n)=\prod_{\ell|n}G_\ell$.  Then for $m$ dividing
$n$ we have the equality $$\Gal(K[n]/K[m])\iso \prod_{\ell\mid (n/m)}G_\ell
\iso G(n/m).$$ Define the derivative operator
$D_\ell\in\Z_p[G(\ell)]$ by $D_\ell=\sum_{i=1}^{\ell}
i\sigma^i_\ell$, where $\sigma_\ell$ is a fixed generator of $G(\ell)$,
and let $D_n=\prod_{\ell|n}D_\ell\in\Z_p[G(n)]$.
One has the easy telescoping identity
$$(\sigma_\ell-1)D_\ell=\ell+1-\Norm_\ell.$$
Choosing a set of coset representatives $S$ for $G(n)\subset\pg(n)$, we
define $$\tilde{\kappa}_n=\sum_{s\in S}sD_n(P[n])\in E(K[n]).$$

\begin{Lem}\label{first derivative property}
The image of $\tilde{\kappa}_n$ in $E(K[n])/I_n E(K[n])$
is fixed by $\pg(n)$.
\end{Lem}
\begin{proof} For each $\ell|n$ we have the equalities in 
$E(K[n])/I_n E(K[n])$
\begin{eqnarray*}(\sigma_\ell-1)D_n(P[n])&=&(\sigma_\ell-1)
D_\ell D_{n/\ell}P[n]\\
&=&-D_{n/\ell}\Norm_\ell (P[n])\\
&=&-a_\ell D_{n/\ell}(P[n/\ell]).
\end{eqnarray*}
Since $a_\ell\in I_\ell\subset I_n$, the claim follows.
\end{proof}

Our assumption that the map $G_K\map{}\Aut(T)$ is surjective guarantees that
$E(K[n])[p]=0$, and so, by the Hochschild-Serre spectral sequence, restriction
gives an isomorphism $$H^1(K,T/I_nT)\map{\res}H^1(K[n],T/I_nT)^{\pg(n)}.$$
If $\delta_n:E(K[n])/I_n E(K[n])\map{}H^1(K[n],T/I_nT)$ is the Kummer map,
we define  $\kappa_n$ to be the unique preimage of
$\delta_n(\tilde{\kappa}_n)$ under restriction.

\begin{Lem}\label{mccallum prop}
The class $\kappa_n\in H^1(K,T/I_nT)$ may be given as an explicit cocycle
as follows.  Let $I_n=p^{M_n}\Z_p$ and fix a $p^{M_n}$-divisor of
$\tilde{\kappa}_n$, $\frac{\tilde{\kappa}_n}{p^{M_n}} \in E(\bar{K})$.
For $\sigma\in G_K$ let $\frac{(\sigma-1)\tilde{\kappa}_n}{p^{M_n}}$ be the
unique $p^{M_n}$-divisor of $(\sigma-1)\tilde{\kappa}_n$ in $E(K[n])$.
Then $$\kappa_n(\sigma)=(\sigma-1)\frac{\tilde{\kappa}_n}{p^{M_n}}-
\frac{ (\sigma-1)\tilde{\kappa}_n}{p^{M_n}}.$$
\end{Lem}
\begin{proof} This is Lemma 4.1 of \cite{mccallum}.
\end{proof}

\begin{Lem} \label{transverseness} Fix $n\in\pn$ and
let $\sel$ denote the Selmer structure of Theorem 
\ref{kolyvagin's theorem II}
 on $T$, so that $H^1_\sel(K,T)=S_p(E/K)$.
Then $\kappa_n\in H^1_{\sel(n)}(K,T/I_nT)$.
\end{Lem}
\begin{proof} The statement that 
$\loc_v(\kappa_n)\in H^1_\sel(K_v,T/I_nT)$ for
$v$ not dividing $n$ is Proposition 6.2 of \cite{gross}.

Assume that $\ell|n$ and let $\lambda$ be the prime of $K$ above $\ell$.
We must show that the restriction of $\kappa_n$ to
$H^1(K[\ell]_{\lambda'},T/I_nT)$ is trivial, where $\lambda'$ is
the unique prime of $K[\ell]$ above $\ell$.
The prime $\lambda$ of $K$ above $\ell$ splits completely in $K[n/\ell]$,
and so $\lambda'$ splits completely in $K[n]$.  Fixing a prime $\lambda''$
of $K[n]$ above $\lambda'$, we have
$K[\ell]_{\lambda'}=K[n]_{\lambda''}$.
Therefore it suffices to show that $\sum_{s\in S}sD_n(\delta_n(P[n]))$
has trivial restriction to $H^1(K[n]_{\lambda''},T/I_nT)$.

Let $$c=\delta_n(P[n])\in H^1_\unr(K[\ell]_{\lambda'},T/I_nT)=
H^1_\unr(K[n]_{\lambda''},T/I_nT)$$ and
extend $\sigma_\ell$ to a generator of $\Gal(K[\ell]_{\lambda'}^\unr/
K_\lambda^\unr)$.
By definition of $I_n$, the Frobenius automorphism,
$\Frob_\lambda\in\Gal(K_\lambda^\unr/K_\lambda)$
acts trivially on $T/I_nT$, and so by Proposition \ref{locally free}
it suffices to show that $(D_\ell c)(\Frob_\lambda)\in T/I_nT$ is zero.
Since $\sigma_\ell$ acts trivially on $T/I_nT$, we have
$$ (D_\ell c)(\Frob_\lambda)=\sum_{i=1}^\ell ic(\Frob_\lambda)=
\frac{\ell(\ell+1)}{2}c(\Frob_\lambda)=0.$$
\end{proof}

\begin{Prop}\label{kolyvagin relations over K}
For every $\ell\mid \lambda\in\pl$ there is an automorphism
$$\chi_\ell:T/I_\ell T\map{}T/I_\ell T$$ such that for $n\ell\in\pn$,
$\chi_\ell(\kappa_n(\Frob_\lambda))=\kappa_{n\ell}(\sigma_\ell)$
as elements of $T/I_{n\ell}T$.
\end{Prop}
\begin{proof} Fix a prime $\lambda'$ of $\bar{K}$ above $\lambda$.
Identify $$T/I_\ell T\iso E[I_\ell]\iso \tilde{E}(\F)[I_\ell]$$ where
$\tilde{E}$ is the reduction of $E$ at $\ell$ and $\F$ is the residue
field of $K$ at $\lambda$.
By Lemma \ref{mccallum prop} (and using the notation
of that lemma) the right hand side is given by the congruence
$$\kappa_{n\ell}(\sigma_\ell)\equiv
-\frac{(\sigma_\ell-1)\tilde{\kappa}_{n\ell}}
{p^{M_{n\ell}}}\pmod{\lambda'}.$$  Combining this with  the
 Euler system relations
and the congruence (\ref{congruence}) gives
\begin{equation*}\kappa_{n\ell}(\sigma_\ell)\equiv
\frac{a_\ell-(\ell+1)\Frob_\ell}{p^{M_{n\ell}}}
\tilde{\kappa}_n \pmod{\lambda'}
\end{equation*}
(see the proof of Proposition 4.4 of \cite{mccallum}).
Define $\chi_\ell$ to be the composition
$$E(K_\lambda)\map{}\tilde{E}(\F)\map{}\tilde{E}(\F)[p^\infty]
\map{ p^{-M_{\ell}}(a_\ell-(\ell+1)\Frob_\ell)}
\tilde{E}(\F)[I_\ell]\map{} E[I_\ell] $$ where the
first arrow is reduction, the second is projection onto the $p$-Sylow
 subgroup, and the last is
the canonical lift to $E(\bar{K}_\lambda)[I_\ell]$.  The action of
$\Frob_\ell$ splits the $p$-Sylow subgroup of $\tilde{E}(\F)$
into cyclic eigencomponents whose lengths are the orders at $p$
of $\ell+1-a_\ell$ and $\ell+1+a_\ell$, it follows that $\chi_\ell$ is
a surjection.  Since $E[I_\ell]$ is defined over $K_\lambda$, the map
$\chi_\ell$ factors through to an isomorphism
$$E(K_\lambda)/I_\ell E(K_\lambda)\map{} E[I_\ell].$$
Identifying $$E(K_\lambda)/I_\ell E(K_\lambda)\iso
H^1(K_\lambda^\unr/K_\lambda, E[I_\ell])\iso E[I_\ell]$$
we obtain  the desired automorphism of $E[I_\ell]$.
\end{proof}

The above proposition shows that the classes $\kappa_n$ almost form
a Kolyvagin system.  Only a slight modification is needed:

\begin{Thm}\label{final ks}
There is a Kolyvagin system $\kappa'$ for $(T,\sel,\pl)$ with
$\kappa_1'=\kappa_1$.
\end{Thm}
\begin{proof} For $n\in\pn$ define an automorphism
$$\chi_n:H^1(K,T/I_n T)\map{} H^1(K,T/I_n T)$$ as follows.
For each $\ell$ dividing $n$, the automorphism $\chi_\ell$ of $T/I_\ell T$
induces an automorphism of $H^1(K,T/I_n T)$, again denoted by
$\chi_\ell$.  It is clear from construction in the proof of
Proposition \ref{kolyvagin relations over K} that the maps $\chi_\ell$
pairwise commute, and we define $\chi_n$ to be the composition of
of the $\chi_\ell$ as $\ell$ runs over all divisors of $n$.
We now define $$\kappa_n'=\chi_n^{-1}(\kappa_n)\otimes_{\ell|n}
{\sigma_\ell}\in H^1_{\sel(n)}(K,T/I_n T)\otimes G_n.$$
\end{proof}

The class $\kappa_1'$ is the image of $\Norm_{K[1]/K}P[1]$ under the Kummer
map $E(K)\map{}H^1(K,T)$, and so is nonzero provided that $L'(E/K,1)\not=0$,
by the results of Gross and Zagier, \cite{gross-zagier}.


\section{Iwasawa theory}\label{Iwasawa}


Fix an elliptic curve $E_{/\Q}$ with good, ordinary reduction at
$p$, and let $K$ be a quadratic imaginary field satisfying the Heegner
hypothesis and with discriminant $\not=-3,-4$ and prime to $p$.
Let $K_\infty/K$ be the anticyclotomic $\Z_p$-extension,
$$\Gamma=\Gal(K_\infty/K)\hspace{1cm}\Lambda=\Z_p[[\Gamma]],$$ 
so that $K_\infty/K$ is characterized as the unique $\Z_p$-extension 
of $K$ such that complex conjugation acts as 
$\tau\sigma\tau=\sigma^{-1}$ for all $\sigma\in\Gamma$.
Fix a topological generator $\gamma\in\Gamma$ 
so that we may identify $\Lambda$
with the power series ring $\Z_p[[T]]$.  Let $K_n$ denote the unique
subfield of $K_\infty$ with $[K_n:K]=p^n$.
Set $$T=T_p(E)\hspace{1cm}V=T\otimes\Q_p\hspace{1cm}A=V/T.$$
We assume throughout that the map $\Gal(\bar{K}/K)\map{}\Aut_{\Z_p}(T)$
is surjective, and that each prime of $K$ above $p$ is totally ramified
in $K_\infty$.

We denote by $f\mapsto f^\iota$
the involution of $\Lambda$ induced by $\gamma\mapsto\gamma^{-1}$.
We regard $\Lambda$ as a $G_K$-module in the obvious
way. The symbol $\Sigma_\Lambda$ will always
be used to indicate a finite set of height-one prime ideals of $\Lambda$, and
$\gp$ will always denote a height-one prime of $\Lambda$.

For a height-one prime $\gp\not=p\Lambda$ of
$\Lambda$, denote by $S_\gp$ the integral closure of $\Lambda/\gp$,
by $\Phi_\gp$ the field of fractions of $S_\gp$, and by $\D_\gp$ the quotient
$\Phi_\gp/S_\gp$. For any $\Z_p$-module $N$, let $N_\gp=N\otimes_{\Z_p}S_\gp$.
If $N$ has a $G_K$-action, we let $G_K$ acts on $N_\gp$ by acting on both
factors in the tensor product, the action on $S_\gp$ being given by the
natural map $G_K\map{}\Lambda\map{}S_\gp$.

Our basic tool for studying the Iwasawa module
$\bT=T_p(E)\otimes\Lambda$ and its cohomology is, following
\cite{mazur-rubin}, to consider the $S_\gp$-module
$T_\gp\iso \bT\otimes_\Lambda S_\gp$
for each height-one prime $\gp$ of $\Lambda$.
The results of Section \ref{KS} allow one to control certain
Selmer groups associated to $T_\gp$, defined using the ideas of
\cite{coates-greenberg}, and from this one may
recover information about the structure of $\Sel_{p^\infty}(E/{K_\infty})$.


\subsection{Kolyvagin systems at height-one primes}
\label{height one section}


Throughout Subsection \ref{height one section} we work with a fixed
height-one prime $\gp\not=p\Lambda$ of $\Lambda$.
Let $\gm$ be the maximal ideal of $S_\gp$.
If  $\mathfrak{d}$ is a generator
for the absolute different of $\Phi_\gp$, the trace from $\Phi_\gp$ to
$\Q_p$ defines a surjective map
$$\D_\gp=\Phi_\gp/ S_\gp\map{\mathfrak{d}^{-1}}
\Phi_\gp/\mathfrak{d}^{-1} S_\gp\map{\mathrm{Tr}}\Q_p/\Z_p$$
whose kernel contains no $ S_\gp$-submodules.  This map induces an isomorphism
of $ S_\gp$-modules $$\Hom_{ S_\gp}(N ,\D_\gp(1))\iso\Hom_{\Z_p}(N ,\mup)$$
for any finitely or co-finitely generated $S_\gp$-module $N$.

If $v$ is a prime of $K$ above $p$, we define $\Fil_v T$ to be the
kernel of the reduction map $T_p(E)\map{}T_p(\tilde{E})$ where
$\tilde{E}$ is the reduction of $E$ at $v$.  Let
$$\Fil_v T_\gp= (\Fil_v T)\otimes S_\gp\hspace{1cm}
\Fil_v V_\gp= (\Fil_v T)\otimes \Phi_\gp.$$  We define the
\emph{ordinary} local condition at $v$, $H^1_\ord(K_v,V_\gp)$,
to be the image of
$$H^1(K_v,\Fil_v V_\gp)\map{}H^1(K_v,V_\gp).$$

\begin{Lem}\label{height one pairing}
There is a perfect $S_\gp$-bilinear pairing
$$e_\gp:T_\gp\times T_\gp\map{}S_\gp(1)$$ which satisfies
$e_\gp(s^\sigma,t^{\tau\sigma\tau})=e_\gp(s,t)^\sigma$ for $s,t\in T_\gp$ and
$\sigma\in G_K$ (here we regard $S_\gp(1)$ as the Tate twist of the module
$S_\gp$ with \emph{trivial} Galois action).  The submodule $\Fil_v T_\gp$
is its own exact orthogonal complement under this pairing.
\end{Lem}
\begin{proof} If $e:T\times T\map{}\Z_p(1)$ is the Weil pairing, we define
$e_\gp$ by $$e_\gp(t_1\otimes\alpha_1, t_2\otimes\alpha_2)=
e(t_1,t_2^\tau)\otimes\alpha_1\alpha_2$$ for $t_i\in T$ and
$\alpha_i\in S_\gp$.  Since $\Fil_v T$ is maximal isotropic under the
Weil pairing, the same is true of $\Fil_v T_\gp$.
\end{proof}

\begin{Def}
Define a Selmer structure $\sel_\gp$ on $V_\gp$ by
$$H^1_{\sel_\gp}(K_v, V_\gp)=\left\{\begin{array}{ll}
H^1_\ord(K_v, V_\gp)& \mathrm{if\ }v\mid p\\
        \\
H^1_\unr(K_v, V_\gp)& \mathrm{else.}
\end{array}\right.$$ We denote also by $\sel_\gp$ the
Selmer structures obtained by  propagating this to $T_\gp$  and to
$A_\gp\iso V_\gp/T_\gp$.
\end{Def}

\begin{Prop}\label{height one proposition}
Fix a positive integer $s$ and a set of primes $\pl\supset
\pl_s(T_\gp) $, and suppose the Selmer triple
$(T_\gp,\sel_\gp,\pl)$ admits a nontrivial Kolyvagin system
$\kappa$. Then $H^1_{\sel_\gp}(K,T_\gp)$ is a free, rank-one
$S_\gp$-module, and
$$H^1_{\sel_\gp}(K,A_\gp)\iso \D_\gp\oplus M_\gp\oplus M_\gp$$
where $M_\gp$ is a finite $S_\gp$-module with
$$\len(M_\gp)\le \len( H^1_{\sel_\gp}(K,T_\gp)/S_\gp\cdot \kappa_1 ).$$
\end{Prop}
\begin{proof} By Theorem \ref{dvr bound}, we need only verify that Hypothesis
H.0--H.5 hold.  Hypothesis H.0 is trivial.  For
Hypothesis H.1, observe that $\Tbar_\gp\iso
E[p]\otimes S_\gp/\gm$.  The action of $G_K$ on $S_\gp/\gm$ factors
through $G_K\map{}\Lambda/(p,\gamma-1)\map{}S_\gp/\gm$, and so is trivial
on the second factor of the tensor product.  Therefore, the surjectivity
of $G_K\map{}\Aut_{\Z_p}(E[p])$ implies that $G_K\map{}\Aut_{S_\gp}
(\Tbar_\gp)$ is also surjective.
For H.2 we take $F=K_\infty(E[p^\infty])$.  
Since $\mu_{p^\infty}\subset F$ and $\Tbar_\gp\iso E[p]\otimes S_\gp/\gm$,
we must show that $H^1(F/K, E[p])=0$.  From the surjectivity of
$G_K\map{}\Aut_{\Z_p}(E[p])$, one may deduce that
$E(K_\infty)[p]=0$ and that
$$H^1(F/K_\infty ,E[p])\iso H^1(K(E[p^\infty])/K, E[p])
\iso H^1(GL_2(\Z_p),\F_p^2)=0$$
(as in Theorem \ref{kolyvagin's theorem II})
and so the claim follows from the exactness of the 
inflation-restriction sequence
$$H^1(K_\infty/K,E(K_\infty)[p])\map{}H^1(F/K,E[p])\map{}
H^1(F/K_\infty,E[p]).$$
Hypothesis H.3 follows from 
Lemma 3.7.1 of \cite{mazur-rubin}
and the fact that the Selmer structure $\sel_\gp$ on $T_\gp$ is obtained by
propagation from $V_\gp$.  Hypothesis H.4 follows from
Lemma \ref{height one pairing}, and H.5 follows from 
the isomorphism $\Tbar_\gp\iso E[p]\otimes S_\gp/\gm$ (with $G_K$ acting
trivially on the second factor).
\end{proof}


\subsection{Kolyvagin systems over $\Lambda$}
\label{ks lambda section}

\begin{Def}
If $M$ is any group on which  $G_K$ acts and $L/K$ is a finite Galois
extension we define the induced representation $$M_{L/K}=\Ind_{L/K}M=
\{f:G_K\map{}M\ |\ f(\sigma x)=f(x)^\sigma\ \forall
x\in G_K,\ \sigma\in G_L\}.$$
This comes equipped with   commuting actions of $G_K$
and $\Gal(L/K)$  defined by $$(f^\sigma)(x)=f(x\sigma)\hspace{1cm}
(\gamma\cdot f)(x)=f(\tilde{\gamma}^{-1}x)^{\tilde{\gamma}}$$
 where $\sigma\in G_K$, $\gamma\in \Gal(L/K)$, and
$\tilde{\gamma}$ is any lift of $\gamma$ to $G_K$.
\end{Def}
We view $\Ind_{L/K}$ as an exact functor from the category of $G_K$-modules
to the category of $G_K$-modules with commuting $\Gal(L/K)$-action.
For $M$ a $G_K$-module, we define
 $G_K$-module maps $$\res:M\map{}M_{L/K} \hspace{1cm}
\cor:M_{L/K}\map{}M$$ by $\res(m)(x)=x\cdot m$ and
$\cor(f)=(\Norm_{L/K}f)(\mathrm{id}_{G_K})$.
Under the canonical identification of Shapiro's lemma  
$H^q(L,M)\iso H^q(K,M_{L/K})$,
$\res$ and $\cor$ induce restriction and corestriction.

\begin{Lem}\label{cohomology induction}
If $F$ is any extension of $L$, there is a canonical isomorphism
$$H^q(F,M_{L/K})\iso \Ind_{L/K} H^q(F,M).$$
\end{Lem}
\begin{proof} This follows from Proposition B.4.2 of \cite{rubin}.
\end{proof}

\begin{Def}\label{Lambda module defs}
Define $\Lambda$-modules $\bT$ and $\bA$ by
$$\bT=\mil \Ind_{K_n/K}T\hspace{1cm}\bA=\dlim\Ind_{K_n/K}A$$
the limits with respect to corestriction and restriction, respectively.
We remark that there is a canonical isomorphism of $\Lambda$ and $G_K$-modules
$\bT\iso T\otimes\Lambda$, where $G_K$ acts on both factors in the
tensor product and $\Lambda$ acts only on the second factor.
\end{Def}

\begin{Prop}\label{limit pairing}
The Weil pairing $e:T\times A\map{}\mup$
induces a perfect $G_K$-equivariant pairing
$$e_\Lambda:\bT\times\bA\map{}\mup$$ satisfying
$e_\Lambda(\lambda \cdot t,a)=e_\Lambda(t,\lambda^\iota\cdot a)$
for $t\in\bT$, $a\in\bA$, and $\lambda\in\Lambda$.
\end{Prop}
\begin{proof} Let $T_n=\Ind_{K_n/K}T$ and $A_n=\Ind_{K_n/K}A$.  
Define a pairing $$\tilde{e}_n:T_n\times A_n\map{}\Ind_{K_n/K}(\mup)$$ by
$\tilde{e}_n(f,f')(x)=e(f(x),f'(x))$.  This pairing is easily seen to be
equivariant for the actions of both $G_K$ and $\Lambda$, and to satisfy
$$\cor(\tilde{e}_n(f,\res(a)))=e(\cor(f),a)$$ for
$f\in T_n$ and $a\in A$.
Define a  pairing $e_n:T_n\times A_n\map{}\mup$
by the composition $$T_n\times A_n\map{}\Ind_{K_n/K}(\mup)\map{\cor}
\mup.$$ Passing to the limit as $n\to\infty$
yields the desired pairing $e_\Lambda$.
\end{proof}

\begin{Def}\label{grading definitions}
If $v$ is a place of $K$ dividing $p$, let $\Fil_v T$
be the kernel of the reduction map
$T\map{}T_p(\tilde{E})$
where $\tilde{E}$ is the reduction of $E$ at $v$.
Define $\Fil_vV=\Fil_vT\otimes\Q_p\subset V $ and
$\Fil_vA=\Fil_vV/\Fil_vT\subset A$.
Define $\Fil_v \bT\subset \bT$ and $\Fil_v \bA\subset \bA$ by
$$\Fil_v \bT=\mil\Ind_{K_n/K}\Fil_v T\hspace{1cm}\Fil_v\bA
=\dlim \Ind_{K_n/K}\Fil_vA.$$
If $N$ is any object for which $\Fil_vN$ is defined,
set $\gr_v N=N/\Fil_v N$.
\end{Def}

The submodules $\Fil_vT$ and $\Fil_vA$ are exact orthogonal complements under
the Weil pairing, and it follows that the same is true of
$\Fil_v\bT$ and $\Fil_v\bA$
under the pairing of Proposition \ref{limit pairing}.

\begin{Def}\label{iwasawa selmer structures}
Define a Selmer structure $\Lsel$ on $\bT$ by taking the unramified condition
at primes of $K$ not dividing $p$, and taking the image of
$$H^1(K_v,\Fil_v\bT)\map{}H^1(K_v,\bT)$$ at primes above $p$.  Define
a Selmer structure, also denoted $\Lsel$, on $\bA$ in a similar manner.
\end{Def}

It follows from the comments following Definition \ref{grading definitions}
that the local conditions $\sel_\Lambda$ on $\bT$ and $\bA$ are
everywhere exact orthogonal complements under the local Tate pairing.

For any height-one prime $\gp\not=p\Lambda$,
the involution of $\Lambda$ induces a map $S_\gp\map{}S_{\gp^\iota}$ which
we continue to denote by $\iota$.  Define a bijection
$\psi:T_\gp\map{}T_{\gp^\iota}$ by
$\psi(t\otimes \alpha)=t^\tau\otimes\alpha^\iota$.  This map satisfies
$$\psi(\lambda x)=\lambda^\iota\psi(x) \hspace{1cm}
\psi(x^\sigma)=\psi(x)^{\tau\sigma\tau}$$ for any $x\in T_\gp$,
$\lambda\in\Lambda$, and $\sigma\in G_K$.  If
$e_\gp:T_\gp\times A_\gp\map{}\mu_{p^\infty}$
is the pairing induced by that of  Lemma \ref{height one pairing} and the
trace form, then $(x,y)\mapsto e_\gp(\psi^{-1}(x),y)$
defines a perfect, $G_K$-invariant pairing
$$T_{\gp^\iota}\times A_\gp\map{}\mu_{p^\infty}$$
satisfying $(\lambda x,y)=(x,\lambda^\iota y)$.
Dualizing the natural map $\bT/\gp^\iota\bT\map{}T_{\gp^\iota}$ and using
the above pairing and the pairing of Proposition \ref{limit pairing},
we obtain a map of $G_K$ and $\Lambda$-modules
\begin{equation}\label{discrete map}
A_\gp\map{}\bA[\gp].
\end{equation}

\begin{Lem}\label{local comparison}
For every height-one prime $\gp\not=p\Lambda$ of $\Lambda$ and every place
$v$ of $K$, the map $\bT\map{}T_\gp$ and the map (\ref{discrete map})
induce maps \begin{eqnarray*}
H^1_\Lsel(K_v,\bT/\gp\bT)&\map{}&
H^1_{\sel_\gp}(K_v,T_\gp)\\
H^1_{\sel_\gp}(K_v,A_\gp)&\map{}&H^1_\Lsel(K_v,\bA[\gp])\end{eqnarray*}
with finite kernels and cokernels which are bounded by constants depending
only on $[ S_\gp:\Lambda/\gp]$.
\end{Lem}
\begin{proof} The case where $v$ does not divide $p$ 
is covered by Lemma 5.3.13 of
\cite{mazur-rubin}, so we assume that $v$ divides $p$.
The kernel of the first map is bounded by the
size of $H^0(K_v,T\otimes  S_\gp/(\Lambda/\gp))$ and so causes no
problems.  To bound the cokernel, it suffices to bound each cokernel
in the composition
\begin{eqnarray}\label{local comparison maps}\lefteqn{
H^1(K_v,\Fil_v(\bT))\map{}H^1(K_v,\Fil_v(T)\otimes \Lambda/\gp)}
\hspace{2cm}\\ \nonumber
& &\map{}H^1(K_v,\Fil_v(T_\gp))\map{}H^1_{\sel_\gp}(K_v,T_\gp).\end{eqnarray}
The cokernel of the first map is controlled by $H^2(K_v,\Fil_v\bT)[\gp]$,
and by local duality  it suffices to bound
$$H^0(K_v,\gr_v\bA)\iso H^0(K_{\infty,v},\gr_vA).$$
This last group is isomorphic to $p$-power torsion of the
reduction of $E$ at $v$ rational over the residue field of $K_v$,
and this is finite.

The cokernel of the second arrow of (\ref{local comparison maps}) is
controlled by $H^1(K_v,T\otimes  S_\gp/(\Lambda/\gp))$.
This group has a bound of the desired sort,
using the fact that $K_v$ has only finitely many extensions of a given
degree.

For the third arrow of (\ref{local comparison maps}) it suffices to
bound the kernel of $$H^1(K_v,\gr_v T_\gp)\map{}H^1(K_v,\gr_v V_\gp),$$
which is controlled by \begin{eqnarray*}
H^0(K_v,\gr_v A_\gp)&\iso& H^0(K_v,(\gr_v A)\otimes S_\gp)\\
&\subset &H^0(K_{\infty,v},(\gr_v A)\otimes S_\gp)\\
&\iso& H^0(K_{\infty,v},\gr_v A)\otimes S_\gp\\
&\iso& H^0(K_v,\gr_vA)\otimes S_\gp
\end{eqnarray*}
where the last isomorphism uses the fact that $K_{\infty,v}/K_v$ is totally
ramified, while $\gr_v A$ is unramified.  Since $H^0(K_v,\gr_v A)$ is
isomorphic to the $p$-power torsion of $E$ defined over the residue
field of $K_v$, we obtain a bound of the desired sort.

Finally, to deal with the second map in the statement of the lemma,
observe that the kernel and cokernel of $H^1(K_v,\bT/\gp^\iota\bT)\map{}
H^1(K_v,T_{\gp^\iota})$ are finite and have bounds of the desired sort, and so
the same is true of
$$H^1(K_v,\bT/\gp^\iota\bT)/ H^1_\Lsel(K_v,\bT/\gp^\iota\bT)
\map{}H^1(K_v,T_{\gp^\iota})/H^1_{\sel_{\gp^\iota}}(K_v,T_{\gp^\iota}).$$
Now apply local duality.
\end{proof}

\begin{Prop}\label{control}
For every height-one prime $\gp\not=p\Lambda$ of $\Lambda$, the
map $\bT\map{}T_\gp$ and the map (\ref{discrete map})
induce maps \begin{eqnarray*}
H^1_\Lsel(K,\bT)/\gp H^1_\Lsel(K,\bT)&\map{}&
H^1_{\sel_\gp}(K,T_\gp)\\
H^1_{\sel_\gp}(K,A_\gp)&\map{}&H^1_\Lsel(K,\bA)[\gp].\end{eqnarray*}
There is a finite set of primes
$\Sigma_\Lambda$ of $\Lambda$ such that for $\gp\not\in\Sigma_\Lambda$
the kernels and cokernels of these maps  are finite and
bounded by a constant depending only on $[ S_\gp:\Lambda/\gp]$.
\end{Prop}
\begin{proof}  This is deduced from the preceeding lemma exactly as in
the proof of Proposition 5.3.14 of \cite{mazur-rubin}.
\end{proof}

\begin{Lem}\label{torsion free}
The $\Lambda$-module $H^1_\Lsel(K,\bT)$ is torsion free.
\end{Lem}
\begin{proof} Let $K_S$ be the maximal extension of $K$ unramified outside
of all primes dividing $p$ and the conductor of $E$.  Then
$H^1_\Lsel(K,\bT)$ is a submodule of $H^1(K_S/K,\bT)$
which has no $\Lambda$-torsion by \cite{perrin-riou-asterisque}
 \S 1.3.3 and the fact that $E(K_\infty)[p]=0$ (by the surjectivity
of $G_K\map{}\Aut(T)$). 
\end{proof}

\begin{Thm}\label{the thing}
Let $X=\Hom(H^1_\Lsel(K,\bA),\Q_p/\Z_p)$ and
suppose that for some $s$  the Selmer triple $(\bT,\Lsel,\pl_s)$
admits a Kolyvagin system $\kolsys$ with $\kappa_1\not=0$.  Then
\begin{enumerate}
\item $H^1_\Lsel(K,\bT)$ is a torsion free, rank one $\Lambda$-module,
\item  there is a torsion $\Lambda$-module
$M$ such that  $\ch(M)=\ch(M)^\iota$ and a pseudo-isomorphism
$$X\sim\Lambda\oplus M\oplus M,$$
\item $\ch(M)$ divides $\ch(H^1_\Lsel(K,\bT)/\Lambda \kappa_1)$.
\end{enumerate}
\end{Thm}
\begin{proof} At every height-one prime $\gp\not=p\Lambda$,
Remark \ref{functorality} and Lemma \ref{local comparison} yield a map
$$\KS(\bT,\Lsel,\pl_s(\bT))\map{}\KS(T_{\gp},\sel_\gp,\pl_s(T_\gp)).$$
Let $\kolsys^{(\gp)}$ be the image of $\kolsys$ under this map.
 It follows from Proposition \ref{control} and Lemma \ref{torsion free}
that $\kappa_1^{(\gp)}$
generates an infinite $S_\gp$-submodule of $H^1_{\sel_\gp}(K,T_\gp)$
for all but finitely many height-one primes.  We let $\Sigma_\Lambda$
be a finite set of height-one primes of $\Lambda$ containing those primes
for which $\kappa_1^{(\gp)}$ has finite order, all prime divisors
of the characteristic ideal of the $\Lambda$-torsion submodule of $X$,
the exceptional set of primes of Proposition \ref{control},
and the prime $p\Lambda$.

Let $\gp\not\in\Sigma_\Lambda$ be a height-one prime.  Since
$\kappa_1^{(\gp)}\not=0$, Proposition \ref{height one proposition}
implies that $H^1_{\sel_\gp}(K,T_\gp)$ is a free rank-one $S_\gp$-module,
and by Proposition \ref{control} so is $H^1_\Lsel(K,\bT)\otimes_\Lambda
S_\gp$.  Part (a) follows immediately from this.  Similarly,
the $S_\gp$-corank of $H^1_{\sel_\gp}(K,A_\gp)$ is one, and it
follows from Proposition \ref{control} that the $\Lambda$-corank
of $H^1_\Lsel(K,\bA)$ is also one.

Now let $f_\Lambda=\ch(H^1_\Lsel(K,\bT)/\Lambda\cdot\kappa_1)$ and
take $\gp\not=p\Lambda$ to be a prime divisor of $f_\Lambda$.  
We want to determine the order of the characteristic ideal of $X$ at 
$\gp$,  following ideas of \cite{mazur-rubin}.  We 
consider an auxilliary ideal $\gq\not\in\Sigma_\Lambda$, 
determine the structure of the Selmer group 
$H^1_{\sel_\gq}(K,A_\gq)$ (or rather the order of the quotient by the
maximal divisible subgroup), and then consider what happens as $\gq$
``approaches'' $\gp$.
Fix a generator $g$ of $\gp$, and let $\gq=(g+p^m)\Lambda$
for some integer $m$.  By Hensel's lemma, for $m\gg 0$ there is
an isomorphism of rings (but not $\Lambda$-modules) 
$\Lambda/\gp\iso \Lambda/\gq$, and  we
take $m$ large enough that this is so.  In particular $\gq$ is a
height-one prime, and increasing $m$ if needed, we assume that $\gq$ is
not contained in $\Sigma_\Lambda$ and does not divide $f_\Lambda$.

Let $d$ denote the Weierstrass degree of $\gp$ (i.e. the $\Z_p$-rank
of $\Lambda/\gp$). We now argue as in the proof of \cite{mazur-rubin}
Proposition 5.3.10.
Using the notation of Proposition
\ref{height one proposition}, Proposition \ref{control}
and the equality of ideals $(\gq,\gp^n)=(\gq,p^{mn})$
imply that one has the equalities
\begin{eqnarray*}
\len_{\Z_p} H^1_{\sel_\gq}(K,T_\gq)/S_\gq \kappa_1^{(\gq)}
&=&\len_{\Z_p} \Lambda/(f_\Lambda,\gq)\\
&=&\len_{\Z_p} \Lambda/(\gp^{\ord_\gp(f_\Lambda)},\gq)\\
&=&m\cdot d\cdot \ord_\gp(f_\Lambda)
\end{eqnarray*}
up to $O(1)$ as $m$ varies.
Similarly, we have
\begin{eqnarray*}2\cdot\len_{\Z_p}M_\gq&=&
\len_{\Z_p}H^1_{\sel_\gq}(K,A_\gq)_{/\mathrm{div}}\\
&=&\len_{\Z_p} (X/\gq X)_{\Z_p-\tors}\\
&=& m\cdot d\cdot \ord_\gp\big(\ch(X_{\Lambda-\tors})\big)
\end{eqnarray*} up to $O(1)$ as $m$ varies.  Here 
$H^1_{\sel_\gq}(K,A_\gq)_{/\mathrm{div}}$ denotes the quotient of
$H^1_{\sel_\gq}(K,A_\gq)$ by its maximal $\Z_p$-divisible submodule.
Applying Proposition \ref{height one proposition} at the prime
$\gq$ and letting $m\to\infty$ we deduce that
\begin{equation}\label{iwasawa inequality}
 \ord_\gp\big(\ch(X_{\Lambda-\tors})\big)\le 2\cdot
\ord_\gp(f_\Lambda).\end{equation}
The case $\gp=p\Lambda$ is dealt with in an entirely similar fashion,
taking $\gq=T^m+p\in \Z_p[[T]]$. This shows that (c) follows from (b).

To prove (b), keep $\gp\not=p\Lambda$ and $\gq$
as above. Fix a pseudo-isomorphism
$$X_{\Lambda-\tors}\sim N\oplus N_\gp$$ where $\ch(N)$ is prime to $\gp$, and
$N_\gp$ is isomorphic to $\bigoplus_i\Lambda/\gp^{e_i}$.
The dual of the second map of Proposition \ref{control} induces
the third arrow of the composition $$N_\gp\otimes_\Lambda S_\gq\map{}
X_{\Lambda-\tors}\otimes_\Lambda S_\gq\map{}
(X\otimes_\Lambda S_\gq)_{\Z_p-\tors}\map{} M_\gq\oplus M_\gq$$
and this composition has finite kernel and cokernel, bounded as $m$ varies.
Fixing a ring isomorphism $S_\gp\iso S_\gq$ 
(which will not be an isomorphism
of $\Lambda$-modules), we may view
$N_\gp\otimes_\Lambda S_\gq$ as an $S_\gp$-module, isomorphic to
$\bigoplus_i S_\gp/p^{me_i}S_\gp$.  Letting $D_m$ denote $M_\gq$, viewed
as an $S_\gp$-module, we now have $S_\gp$-module maps
$$\bigoplus_i S_\gp/p^{me_i}S_\gp\map{}D_m\oplus D_m$$
with kernels and cokernels bounded as $m$ varies.  An elementary argument
shows that for a given $e$, $\{ i\mid e_i=e\}$ has an even number of
elements.  The case $\gp=p\Lambda$ is dealt with similarly, again
taking $\gq=T+p^m\in\Z_p[[T]]$.

The functional equation $\ch(M)=\ch(M)^\iota$ follows from the functional
equation of \cite{nek-selmer}
$$\ch(X_{\Lambda-\tors})=\ch(X_{\Lambda-\tors})^\iota.$$
\end{proof}


\subsection{The anticyclotomic Euler system}
\label{iwasawa heegner points}


We retain  all notation and assumptions from the introduction to Section 2,
and in addition assume that $p$ does not divide the class number of $K$.
Denote by $K_k$ the subfield of the anticyclotomic extension  $K_\infty/K$
satisfying $[K_k:K]=p^k$. By the assumption on the class number of $K$,
$K_\infty/K$ is linearly disjoint from the Hilbert
class field $K[1]$, and $K_k$ is the maximal $p$-power subextension
of $K[p^{k+1}]/K$.
Let $\bT$ and $\bA$ be as in Definition \ref{Lambda module defs} and let
$\sel_\Lambda$ be the Selmer structure of
Definition \ref{iwasawa selmer structures}. 
Define $\pl=\pl_1(\bT)$.
The majority of this subsection is devoted to the proof of the 
following theorem.

\begin{Thm}\label{heegner kolyvagin system}
There exists a Kolyvagin system $\kolsys^\Heeg\in
\KS(\bT,\Lsel,\pl)$ such that $\kappa^\Heeg_1 \in H^1_\Lsel(K,\bT)$
is nonzero.
\end{Thm}

For $n\in\pn$ let $K_k[n]$ be the compositum
of $K_k$ and $K[n]$,  and let $K_\infty[n]$ be the union over all $k$
of $K_k[n]$. There is a canonical isomorphism
$$(\cO_K/p\cO_K)^\times/(\Z/p\Z)^\times
\iso\Gal(K[np^{k+1}]/K_k[n])$$ and we denote this group by $\Delta$.
Let $\delta=|\Delta|$.
If $p$ is split in $K$ we let $\sigma$ and $\sigma^*$ denote the
Frobenius automorphisms in $\pg(n)=\Gal(K[n]/K)$ of the primes above $p$.
Define  $\gamma_k, \Phi\in\Z_p[\pg(n)]$ by the formulas
\begin{eqnarray*}
\Phi&=&\left\{\begin{array}{ll}
(p+1)^2-a_p^2&\mathrm{inert\ case}\\
(p-a_p\sigma+\sigma^2)(p-a_p\sigma^*+\sigma^{*2})
&\mathrm{split\ case}\end{array}\right.\\
\gamma_0&=&\left\{\begin{array}{ll}
a_p&\mathrm{inert\ case}\\
a_p-\sigma-\sigma^*&\mathrm{split\ case}\end{array}\right.\\
\gamma_1&=&a_p\gamma_0-\delta\\
\gamma_k&=&a_p\gamma_{k-1}-p\gamma_{k-2}\hspace{.5cm}\mathrm{for\ }k>1
\end{eqnarray*}
where split and inert refer to the behavior of the rational prime $p$ in $K$.

Define points $P_k[n]\in E(K_k[n])$ by
$$P_k[n]=\Norm_{K[np^{k+1}]/K_k[n]}P[np^{k+1}]$$ for $k\ge 0$,
and denote by $H_k[n]$ the $\Z_p[\Gal(K_k[n]/K)]$-submodule of
$E(K_k[n])\otimes\Z_p$ generated by $P[n]$ and $P_j[n]$ for all $j\le k$.
It follows from Section 3.1 of \cite{pr87} that one has the relations
\begin{eqnarray*}
P_0[n]&=&\gamma_0 P[n]\\
\Norm_{K_{k+1}[n]/K_k[n]}P_{k+1}[n]&=&
\left\{\begin{array}{ll}
a_p P_k[n]-P_{k-1}[n]&\mathrm{for\ }k>0\\
\gamma_1 P[n]&\mathrm{for\ }k=0\end{array}\right.\\
\Norm_{K_k[n\ell]/K_k[n]}P_k[n\ell]&=&a_\ell P_k[n],
\end{eqnarray*}
and an easy inductive argument using the first two of these 
relations shows that
$$\Norm_{K_k[n]/K[n]}P_k[n]=\gamma_k P[n]
\hspace{.5cm}\mathrm{for\ }k\ge 0.$$
We observe also that the norm from $K_{k+1}[n]$ to $K_k[n]$ takes
$H_{k+1}[n]$ into $H_k[n]$, and so we may define
for every $n\in \pn$ a $\Lambda[\pg(n)]$-module
$$\bH[n]=\mil H_k[n].$$

\begin{Lem}\label{universal norm lemma}
If $M$ is any finitely generated $\Z_p[\pg(n)]$-module, the intersection
of $\gamma_k M$ for $k\ge 1$ is equal to $\Phi M$.
\end{Lem}
\begin{proof} This is Corollaire 5 of section 3.3 of \cite{pr87}.
\end{proof}

\begin{Lem}\label{euler system} There exists a family
$$\{Q[n]=\mil Q_k[n]\in \bH[n]\}_{n\in\pn}$$
such that $Q_0[n]=\Phi P[n]$, and for any $n\ell\in\pn$
$$\Norm_{K_\infty[n\ell]/K_\infty[n]}Q[n\ell]=a_\ell Q[n].$$
\end{Lem}
\begin{proof} Fix an $n\in\pn$ and
let $\tilde{H}_k$ be the free $\Z_p[\Gal(K_k[n]/K)]$-module
on generators $\{x, x_j\mid 0\le j\le k\}$, modulo relations of the form
\begin{enumerate}
\item  $x$ is fixed by $\Gal(K_k[n]/K[n])$, and
$x_j$ is fixed by $\Gal(K_k[n]/K_j[n])$ for every $j\le k$,
\item For $j>1$, $\Norm_{K_{j}[n]/K_{j-1}[n]}x_{j}=a_p x_{j-1}-x_{j-2}$,
\item $\Norm_{K_1[n]/K_0[n]}x_{1}=\gamma_1 x$, and $x_0=\gamma_0 x$.
\end{enumerate}
Then for each $j\le k$, \begin{equation}\label{universal relation}
\Norm_{K_j[n]/K_0[n]}x_j=  \gamma_j x.\end{equation}
 There is a natural inclusion $\tilde{H}_k\map{}
\tilde{H}_{k+1}$ and a natural norm $\tilde{H}_{k+1}\map{}
\tilde{H}_k$.  By Lemma \ref{universal norm lemma} and the relation
(\ref{universal relation}), $\Phi x\in\tilde{H}_0$ is
a norm from every $\tilde{H}_k$.

Let $y\in \tilde{\bH}=\mil\tilde{H}_k$
be a lift of $\Phi x$, and define, for any $m\mid n$,
$Q[m]$ to be the image of
$y$ under the map $\phi(m):\tilde{\bH}\map{}\bH[m]$ which sends $x_k\mapsto
P_k[m]$ and $x\mapsto P[m]$. For any $m\ell\mid n$, the diagram
$$\xymatrix{
{\tilde{\bH}} \ar[r]^{\phi(m\ell)}\ar[d]^{a_\ell} & {\bH[m\ell]}\ar[d]\\
{\tilde{\bH}} \ar[r]^{\phi(m)} & {\bH[m]}}$$
commutes, where the right vertical arrow is the norm from
$K_\infty[m\ell]$ to $K_\infty[m]$, and so we obtain
 a family $\{Q[m]\}_{m\mid n}$ with the desired properties.

An easy argument shows that the $\Lambda$-module of such ``partial''
families (i.e. where $m$ runs through divisors of a fixed $n$)
is compact, and so the inverse limit over all $n\in\pn$ is nonempty.
\end{proof}

Fix a family $Q[n]$ as in the lemma.
Exactly as in Section \ref{Heegner points}, we fix a generator 
$\sigma_\ell$ of $G(\ell)$ for every $\ell\in\pl$ and define 
derivative operators
$$D_n\in \Z_p[G(n)]\subset \Lambda[G(n)].$$ Fix a set of coset
representatives $S$ of $G(n)\subset\pg(n)$.  Let
$$\tilde{\kappa}_n=\sum_{s\in S} sD_n Q[n] \in \bH[n].$$
For $\ell\in\pl$, the ideal $I_\ell\subset \Z_p$ is generated by $\ell+1$
and $a_\ell$, and
the image of $\tilde{\kappa}_n$ in $\bH[n]/I_n \bH[n]$ is fixed by
$\pg(n)$ (see Lemma \ref{first derivative property}).

The Kummer map $\delta_k(n):E(K_k[n])\otimes \Z_p\map{}H^1(K_k[n],T_p(E))$
induces a map 
\begin{eqnarray*}
\delta(n)=\mil\delta_k(n):\bH[n]&\map{}&\mil H^1(K_k[n],T_p(E))\\
&\iso& H^1(K[n],\bT)\end{eqnarray*}
and we define $\kappa_n$ to be the unique preimage of
$\delta(n)(\tilde{\kappa}_n)$ under the isomorphism
$$H^1(K,\bT/I_n\bT)\map{}H^1(K[n],\bT/I_n\bT)^{\pg(n)}$$
(the bijectivity being a consequence of $$H^0(K[n],\bT/I_n\bT)\iso
\mil H^0(K_k[n],E[I_n])=0,$$ since $E$ has no $p$-torsion defined over any
abelian extension of $K$).

\begin{Lem}
For every $n\in\pn$, $\kappa_n\in H^1_{\Lsel(n)}(K,\bT/I_n\bT)$.
\end{Lem}
\begin{proof}
The proof that the localization of $\kappa_n$ at primes of $K$ dividing
$n$ lies in the transverse subspace is exactly as in the proof of Lemma
\ref{transverseness}.

It remains to show that at every prime $v$ of $K$ not dividing $n$,
the localization of $\kappa_n$ at $v$ is contained in
$H^1_\Lsel(K_v,\bT/I_n\bT)$, the image of the map
$$H^1_{\Lsel}(K_v,\bT)\map{}H^1(K_v,\bT/I\bT).$$
Fix a prime $v$ of $K$ not dividing $n$
and let $w$ be a prime of $K[n]$ above $v$.

Case (i), $v\not|\ pN$. We first observe that
$$H^1_{\Lsel}(K_v,\bT/I_n\bT)=H^1_\unr(K_v,\bT/I_n \bT).$$
Indeed, since $\Gal(K_v^\unr/K_v)$ has cohomological dimension one,
the map $$H^1_\unr(K_v,\bT)\map{} H^1_\unr(K_v,\bT/I_n \bT)$$ is surjective.
Using the injectivity of
torsion points in the reduction of $E$ at $w$, the image
of the Kummer map 
$$\delta_k(N):H_k[n]\map{}\bigoplus_{w'|w}H^1(K_k[n]_{w'},T)\iso
H^1(K[n]_w,\Ind_{K_k/K}T)$$ is unramified, and passing to the limit
shows that the image of
$$\delta(n):\bH[n]\map{}H^1(K[n],\bT)\map{}H^1(K[n]_w,\bT)$$
 is unramified at $w$. Therefor $\delta(n)(\tilde{\kappa}_n)$ is
 unramified, and so also is $\kappa_n$.

Case (ii), $v|N$.  In this case the Heegner hypothesis implies that the
prime $w$ is finitely decomposed in $K_\infty[n]$. 
Proposition B.3.4 of \cite{rubin} gives the equality
$$H^1(K_v,\bT)=H^1_\unr(K_v,\bT),$$
and we must therefore show that
$\loc_v(\kappa_n)$ is in the image of
$$H^1(K_v,\bT)\map{}H^1(K_v,\bT/I_n \bT).$$
On the other hand, the restriction of $\kappa_n$ to $H^1(K[n]_w,\bT/I_n\bT)$
comes from $H^1(K[n]_w,\bT)$ (namely from the localization
of $\delta(n)(\tilde{\kappa}_n)$)
and so it suffices to check that the right vertical arrow in the  exact
and  commutative diagram
$$\xymatrix{
H^1(K_v,\bT)\ar[r]\ar[d]&H^1(K_v,\bT/I_n\bT)\ar[r]\ar[d]&H^2(K_v,\bT)\ar[d]\\
H^1(K[n]_w,\bT)\ar[r]&H^1(K[n]_w,\bT/I_n\bT)\ar[r] & H^2(K[n]_w,\bT)}$$
is an injection.
Applying local duality and Shapiro's lemma, 
this is equivalent to the surjectivity of the norm map
$$\bigoplus_{w'|w}E(K_{\infty}[n]_{w'})[p^\infty]\map{}
\bigoplus_{v'|v}E(K_{\infty,v'})[p^\infty],$$
which is a consequence of the observation that
the degree of $K_{\infty}[n]_{w'}$ over  $K_{\infty,v'}$ is prime to $p$.
Indeed, any intermediary extension $$K_{\infty,v'}\subset F\subset
K_{\infty}[n]_{w'}$$ of $p$-power order over $K_{\infty,v'}$
 would be contained in the union
of all unramified $p$-power extensions of $K_v$, and this union is
$K_{\infty,v'}$, the unique $\Z_p$-extension of $K_v$.

Case(iii), $v|p$.  For each prime $w$ of $K[n]$, fix an extension of $w$
to $\bar{K}$ and denote by  $\Fil_w (T)$ the kernel of the reduction map
$T\map{}T_p(\tilde{E})$ at that place. Set $\gr_w(T)=T/\Fil(T)$.
Let $$\Fil_w(\bT)=\Fil_w( T)\otimes\Lambda\subset \bT\hspace{1cm}
\gr_w(\bT)=\bT/\Fil_w(\bT)$$ and
define
$$H^1_\ord(K[n]_w,\bT)=\mathrm{image}\big( H^1(K[n]_w, \Fil_w (\bT))
\map{}H^1(K[n]_w,\bT)\big).$$

We first claim that the image of the composition
$$\bH[n]\map{\delta(n)}H^1(K[n],\bT)\map{}H^1(K[n]_w,\bT)$$ lies in
$H^1_\ord(K[n]_w,\bT)$.  To see this, let $L_k=K_k[n]_{w}$
and consider the composition
$$H_k[n]\map{}H^1(L_k,T)\map{}H^1(L_k,\gr_w(T))
\map{}H^1(L_k^\unr,\gr_w(T)).$$ It is clear from the definition of
the Kummer map that this composition is trivial, and so any $Q_k\in
H_k[n]$ yields a class in the kernel of the final arrow,
$$H^1(L_k^\unr/L_k,\gr_w(T))\iso \gr_w(T)/(\Frob-1)\gr_w(T)\iso
\tilde{E}(\F[n])[p^\infty]$$ where $\F[n]$ is the residue field of $K[n]_w$,
and using the fact that $L_k/K[n]_w$ is totally ramified.
If the point $Q_k$ can be lifted to a universal norm in $\bH[n]$, then
this class can be lifted to an element of the $p$-adic Tate module of
the finite group $\tilde{E}(\F[n])[p^\infty]$, which is trivial.  The
composition $$\bH[n]\map{}H^1(L_k,T)\map{}H^1(L_k,\gr_w(T))$$ is therefore
trivial, and the claim follows.

The above shows that the restriction of $\kappa_n$ to $H^1(L_0,\bT/I_n\bT)$
lies in the image of $H^1(L_0,\Fil_w(\bT))$ under the natural map.
For brevity, we write $$\bT^+=\Fil_w(\bT)\hspace{2cm}\bT^-=\gr_w(\bT).$$
Consider the exact and commutative diagram
$$\xymatrix{ H^1(K_v,\bT^+/I_n\bT^+)\ar[r]\ar[d]&H^1(K_v,\bT/I_n\bT)
\ar[r]\ar[d]& H^1(K_v,\bT^-/I_n\bT^-)\ar[d]\\
H^1(L_0,\bT^+/I_n\bT^+)\ar[r]&H^1(L_0,\bT/I_n\bT)\ar[r]&
H^1(L_0,\bT^-/I_n\bT^-).}$$
The image $\loc_v(\kappa_n)$ in the lower right corner is trivial, and the
kernel of the right hand vertical map
is $$\mil H^1(K_\infty[n]_w/K_{\infty,v},
\tilde{E}(\F[n])[I_n])$$ where the inverse limit
is respect to multiplication by $p$.  This is clearly zero,
and so we may choose an $\alpha\in H^1(K_v,\bT^+/I_n\bT^+)$ which lifts
$\kappa_n$.  It is easily seen that the bottom left horizontal arrow
is injective, and so the image of $\alpha$ under the left vertical arrow
is the unique lift to $H^1(L_0,\bT^+/I_n\bT^+)$ of the restriction
of $\kappa_n$ to $H^1(L_0,\bT/I_n\bT)$, which is already known to be in
the image of $H^1(L_0,\bT^+)$.  In other words, in the diagram
$$\xymatrix{ H^1(K_v,\bT^+)\ar[r]\ar[d]&H^1(K_v,\bT^+/I_n\bT^+)\ar[r]\ar[d]&
H^2(K_v,\bT^+)\ar[d]\\
H^1(L_0,\bT^+)\ar[r]&H^1(L_0,\bT^+/I_n\bT^+)\ar[r]&H^2(L_0,\bT^+)}$$
the image of $\alpha$ in the lower right corner is trivial.

To complete
the proof, we need only show that the right vertical arrow is injective.
By local duality, the injectivity of this map is equivalent
to surjectivity of the norm map $$\tilde{E}(\F[n])[p^\infty]\map{}
\tilde{E}(\F)[p^\infty]$$ (where $\F$ is the residue field of $K_v$), and
this follows from $$H^1(\F[n]/\F,\tilde{E}(\F[n])[p^\infty])\hookrightarrow
H^1(\F,\tilde{E}[p^\infty])\iso \tilde{E}[p^\infty]/(\Frob-1)
\tilde{E}[p^\infty]=0$$ and the fact that the Herbrand quotient of a finite
cyclic group acting on a finite module is equal to 1.
\end{proof}

Fix $n\ell\in\pn$ and let $\lambda$ be the prime of $K$ above $\ell$
and $\lambda'$ a fixed place of $\bar{K}$ above $\lambda$.  Such a choice
gives a canonical extension of each prime $w$ of $K_k$ above $\lambda$
to a prime $w'$ of $K_k[n\ell]$.  Namely the unique place which restricts
to $w$ in $K_k$ and to $\lambda'$ in $K[n\ell]$ (recall that $\lambda$ splits
completely in $K_\infty[n\ell]$).  This determines a map of
$\Lambda$-modules
\begin{equation}\label{first residue}
\Psi:\bH[n\ell]\map{} \mil \bigoplus_{w}\tilde{E}(\F_w)
\end{equation}
where the limit is over $k$, the sum is over primes of
$K_k$ above $\lambda$, and $\F_w$ is the residue field of $w$.
Each summand is canonically identified with the points of $\tilde{E}$ rational
over the residue field of $K$ at $\lambda$ (which we denote by $\F_\lambda$),
 and $\Lambda$ acts by permuting summands.
The module on the right hand
side of (\ref{first residue}) comes equipped with a natural involution
$\Frob_\ell$ which acts as the nontrivial automorphism of $\F_w/\F_\ell$
on each summand.  The action of $\Frob_\ell$ commutes with the action of
$\Lambda$.

\begin{Lem} For any $t\in \Lambda[\pg(n\ell)]$,
$\Psi(t\cdot Q[n\ell])=\Frob_\ell\Psi(t\cdot Q[n]).$
\end{Lem}
\begin{proof} Exactly as in (\ref{congruence}), for any prime $w'$
of $K_k[n\ell]$ above $\ell$ and any $j\le k$, we have
$$P_j[n\ell]\equiv \left(\frac{w'}{K_k[n\ell]/\Q}\right)P_j[n]
\pmod{w'}$$
which implies that for any $t\in\Z_p[\Gal(K_k[n\ell]/K)]$
$$\Psi_k(t\cdot P_j[n\ell])=\Frob_\ell\Psi_k(t\cdot P_j[n])$$
where $\Psi_k:H_k[n\ell]\map{}\bigoplus_{w}\tilde{E}(\F_w)$
(the sum is over prime of $K_k$ above $\lambda$) is the map $\Psi$ at
finite levels.
By construction of $Q[n\ell]$ there are elements
$$\{t_j\in\Z_p[\Gal(K_k[n\ell]/K)]  \mid 0\le j\le k\}$$
such that $Q_k[m]=\sum_{j=0}^k t_j P_j[m]$ for every $m\mid n\ell$
(in particular the $t_j$'s do not depend on $m$), and the
claim follows easily. 
\end{proof}

Our choice of $\lambda'$ also fixes an isomorphism
\begin{equation}\label{second residue}
 E[I_{n\ell}]\otimes\Lambda\iso \bT/I_{n\ell}\bT
\iso \mil \bigoplus_{w}\tilde{E}(\F_w)[I_{n\ell}]
\end{equation}
which sends elements of the form $P\otimes\sigma$ to the reduction of
$P$ at $\lambda'$ living in the summand attached to the prime
$\sigma\lambda'$ of $K_\infty$.  Exactly as in the proof of Proposition
\ref{kolyvagin relations over K} we have an explicit description
of the image of $\kappa_{n\ell}\otimes\sigma_\ell$ under the
isomorphism $$H^1_\s(K_\lambda,\bT/I_{n\ell}\bT)\otimes G_\ell\map{}
\bT/I_{n\ell}\bT\iso E[I_{n\ell}]\otimes\Lambda
\map{} \mil \bigoplus_{w}\tilde{E}(\F_w)[I_{n\ell}],$$
namely $$\kappa_{n\ell}\otimes\sigma_\ell\mapsto
\Psi\left(-\frac{(\sigma_\ell-1)\tilde{\kappa}_{n\ell}}{p^{M_{n\ell}}}
\right)$$
where $p^{M_{n\ell}}\Z_p=I_{n\ell}$, and the right hand side is interpreted
as the image of the unique $p^{M_{n\ell}}$-divisor of
$-(\sigma_\ell-1)\tilde{\kappa}_{n\ell}$ in $H[n\ell]$ under the map
(\ref{first residue}) (uniqueness follows from the fact that our assumptions
on $E$ imply that $E$ has no $p$-torsion defined over any abelian
extension of $K$).

\begin{Lem}
$$\Psi\left(-\frac{(\sigma_\ell-1)\tilde{\kappa}_{n\ell}}{p^{M_{n\ell}}}
\right)= \frac{a_\ell-(\ell+1)\Frob_\ell}{p^{M_{n\ell}}}
\Psi(\tilde{\kappa}_n).$$
\end{Lem}
\begin{proof} In $\bH[n\ell]$ we have the equalities
\begin{eqnarray*}
-\frac{(\sigma_\ell-1)\tilde{\kappa}_{n\ell}}{p^{M_{n\ell}}}&=&
-\frac{\sum_{s\in S}(\ell+1-\Norm_\ell)s D_n Q[n\ell] }
{p^{M_{n\ell}}}\\
&=& \sum_{s\in S}sD_n\left(\frac{a_\ell}{p^{M_{n\ell}}} Q[n]
-\frac{\ell+1}{p^{M_{n\ell}}} Q[n\ell]\right)\\
&=& \frac{a_\ell}{p^{M_{n\ell}}} \tilde{\kappa}_n-
\frac{\ell+1}{p^{M_{n\ell}}}\sum_{s\in S}sD_n Q[n\ell].
\end{eqnarray*} Now apply the preceeding lemma.
\end{proof}

As in the proof of Lemma \ref{kolyvagin relations over K}, we define a map
$\chi_\ell$ as the composition
$$ \mil\bigoplus_w E(K_{k,w})
\map{} \mil \bigoplus_w\tilde{E}(\F_w)[p^\infty]
\map{ } \mil \bigoplus_w\tilde{E}(\F_w)[I_\ell]
\iso \bT/I_\ell\bT$$
where the second arrow is given by the action of
$\frac{a_\ell-(\ell+1)\Frob_\ell}{p^{M_{\ell}}}$.
This map factors through
$$\big( \mil\bigoplus_w E(K_{k,w})\big)\otimes_\Lambda \Lambda/I_\ell
\iso H^1_\f(K_\lambda,\bT/I_\ell \bT)\iso  \bT/I_\ell\bT, $$
where the first map is the Kummer map and the second is evaluation
of cocycles at Frobenius.
The resulting automorphism of $\bT/I_\ell\bT$ is again called $\chi_\ell$,
and satisfies
$$\chi_\ell(\kappa_n(\Frob_\lambda))=
\frac{a_\ell-(\ell+1)\Frob_\ell}{p^{M_{n\ell}}}\Psi(\tilde{\kappa}_n)=
\kappa_{n\ell}(\sigma_\ell).$$
The classes $\kappa_n$ may now be modified exactly as in
Theorem \ref{final ks} to produce a Kolyvagin system $\kappa^\Heeg\in
\KS(\bT,\Lsel, \pl)$ with $\kappa_1^\Heeg=\kappa_1$.

Now we turn our attention to the proof that $\kappa^\Heeg_1$ is nontrivial.
Let $$H_k\subset E(K_k)\otimes \Z_p$$ be the $\Lambda$-submodule generated by
$\Norm_{K[1]/K} P[1]$ and $\Norm_{K_k[1]/K_k}P_j[1]$ for $0\le j\le k$, 
and let $\bH=\mil H_k$.
Since $\kappa_1^\Heeg$ is the image of $\tilde{\kappa}_1$ under the injective
Kummer map $\bH\map{}H^1(K,\bT),$
to complete the proof of Theorem \ref{heegner kolyvagin system}
it suffices to prove the following

\begin{Thm}\label{heegner module}
The $\Lambda$-module $\bH$ is free of rank one, generated by
$\tilde{\kappa}_1$.
\end{Thm}
\begin{proof} By the main result of \cite{cornut}, one of the points
$\Norm_{K_k[1]/K_k}P_k[1]$ has infinite order, and so Proposition
10 of section 3 of \cite{pr87} implies that $\bH$ is free of rank one.
We show that $\tilde{\kappa}_1$ is a generator.

Recall the construction of $\tilde{\kappa}_1$.  There is a canonical
decomposition $$\Gal(K_k[1]/K)\iso \Gamma_k\times \pg$$
where $\Gamma_k=\Gal(K_k/K)$ and $\pg=\pg(1)$ is the ideal class group
of $K$ (which has no $p$-torsion by assumption).  We let
$\Norm_{\pg}$ be the norm element in $\Z_p[\pg]\subset\Lambda[\pg]$.
Let $\tilde{H}_k$ be the $\Z_p[\Gamma_k\times
\pg]$-module defined in the proof of
Lemma \ref{euler system} (with $n=1$),
and let $\tilde{\bH}=\mil \tilde{H}_k$, the limit with respect to the
norm maps.  We may choose an element $y\in\tilde{\bH}$ which lifts
$\Phi x\in \tilde{H}_0$.  
Let $$x_j^\pg=\Norm_{\pg(1)}x_j\in \tilde{H}_k^\pg
\hspace{1cm}y^\pg=\Norm_{\pg(1)}y\in\tilde{\bH}^\pg$$
(including the case where $j$ is the empty subscript).

  We have the commutative diagram
in which all arrows are surjective and the vertical arrows are
$\Norm_{\pg(1)}$
$$\xymatrix{
{\tilde{\bH}}\ar[r]\ar[d]&{\bH[1]}\ar[d]\\
{\tilde{\bH}^\pg}\ar[r]&{\bH}.}$$
The top arrow takes $x_j$ to $P_j[1]$, and the bottom arrow
takes $x_j^\pg$ to $\Norm_{\pg(1)}P_j[1]$ and $y^\pg$ to $\tilde{\kappa}_1$.

Fix a topological generator $\gamma\in\Gamma$.
By Nakayama's lemma we will be done once we show that $$\tilde{\bH}^\pg=
\Lambda y^\pg+ (\gamma-1)\tilde{\bH}^\pg.$$
This is immediate from the following two lemmas.

\begin{Lem} Let $\aug:\Z_p[\pg(1)]\map{}\Z_p$ be the augmentation map.
The image of the natural map
$\tilde{\bH}^\pg\map{}\tilde{H}^\pg_0$ is a free rank-one $\Z_p$-module
generated by  $\aug(\Phi)x^\pg$, the image of $y^\pg$.
\end{Lem}
\begin{proof} 
The $\Z_p$-module $\tilde{H}^\pg_0$ is free of rank one, generated by
$x^\pg$, and one has the relations $$\Norm_{K_k/K}(x_k^\pg)
=\aug(\gamma_k)x^\pg.$$
Lemma \ref{universal norm lemma} implies that
$\cap_{k>0}\aug(\gamma_k)\Z_p= \aug(\Phi)\Z_p$,
and an elementary argument using the recursion relation defining
$\gamma_k$ shows that $\aug(\gamma_k)\Z_p= \aug(\Phi)\Z_p$
for $k\gg0$. The claim follows.
\end{proof}

\begin{Lem}
The map of the preceeding lemma induces an isomorphism
$$\tilde{\bH}^\pg/(\gamma-1)\tilde{\bH}^\pg\map{}\aug(\Phi)\tilde{H}_0^\pg.$$
\end{Lem}
\begin{proof}
We have seen that it is a surjection, so suppose $h=\mil h_k$
is in the kernel of $\tilde{\bH}^\pg\map{}\tilde{H}_0^\pg.$
The $\Lambda$-module $\tilde{H}_k^\pg$ is generated by
$x^\pg_k$ and $x^\pg_{k-1}$, and so $h_k$ may be written in the form
$$h_k=\alpha_k x^\pg_k+\beta_k x^\pg_{k-1}+(\gamma-1)z_k$$ for $\alpha_k$
and $\beta_k$ in $\Z_p$.  Taking the norm to $\tilde{H}_0^\pg$ and using the
fact that $x^\pg$ has infinite order yields
$$0= \alpha_k\aug(\gamma_k)+p\beta_k \aug(\gamma_{k-1})$$
and so $$\aug(\gamma_k)h_k\in \beta_k s_k+(\gamma-1)\tilde{H}_k^\pg$$
where $s_k=-p\cdot\aug(\gamma_{k-1})x^\pg_k+\aug(\gamma_k)x^\pg_{k-1}.$
The recursion relation for the $\gamma_j$'s and the norm relations
for the $x_j$'s imply that the norm from $\tilde{H}^\pg_{k+1}$ to
$\tilde{H}^\pg_k$ to takes $s_{k+1}$ to $p\cdot s_k$.
If we take $k$ large enough that $\aug(\gamma_\ell)=\aug(\Phi)$ for
all $\ell\ge k$, and take $\ell\gg k$
$$\aug(\gamma_k)h_k=\aug(\gamma_\ell)\Norm_{\ell/k}h_\ell
\in \beta_\ell p^{\ell-k}s_k+(\gamma-1)\tilde{H}^\pg_k.$$
Letting $\ell\to\infty$ shows that $h_k\in(\gamma-1)\tilde{H}_k^\pg$
for every $k$ and the claim follows.
\end{proof}

This completes the proof of Theorems \ref{heegner kolyvagin system}
and \ref{heegner module}.\end{proof}

\bibliographystyle{amsalpha}

\end{document}